\documentclass[12pt]{article}
\usepackage{amsthm}
\usepackage{amsmath}
\usepackage{amssymb}
\usepackage{latexsym}
\usepackage{enumitem}
\usepackage{comment}
\usepackage{xcolor}
\usepackage{authblk}
\usepackage{array}
\usepackage{multirow}
\usepackage{float}

\newtheorem{definition}{Definition}[section]
\newtheorem{theorem}{Theorem}[section]
\newtheorem{lemma}{Lemma}[section]
\newtheorem{proposition}{Proposition}[section]

\newtheorem{example}{Example}[section]
\newtheorem{corollary}{Corollary}[section]

\newtheorem{remark}{Remark}[section]

\usepackage[backend=biber,style=apa,citestyle=authoryear,uniquename=false,uniquelist=false]{biblatex}
\usepackage[colorlinks,linkcolor=blue,anchorcolor=blue,citecolor=red]{hyperref}

\DeclareMathOperator{\esssup}{ess\hspace{0.1em}sup}

\title{Information-theoretic classification of the cutoff phenomenon in Markov processes}
\author[1]{Youjia Wang\thanks{Email: e1124868@u.nus.edu}}
\author[2]{Michael C.H. Choi\thanks{Email: mchchoi@nus.edu.sg}}
\affil[1]{Department of Statistics and Data Science, National University of Singapore, Singapore}
\affil[2]{Department of Statistics and Data Science and Yale-NUS College, National University of Singapore, Singapore}
\date{\today}

\begin{document}

\maketitle

\begin{abstract}
    We investigate the cutoff phenomenon for Markov processes under information divergences such as $f$-divergences and R\'enyi divergences. We classify most commonly used divergences into four types, namely $L^2$-type, $\mathrm{TV}$-type, separation-type and $\mathrm{KL}$ divergence, in which we prove that the cutoff phenomenon are equivalent and relate the cutoff time and window among members within each type. To justify that this classification is natural, we provide examples in which the family of Markov processes exhibit cutoff in one type but not in another. We also establish new product conditions in these settings for the processes to exhibit cutoff, along with new results in non-reversible or non-normal situations. The proofs rely on a functional analytic approach towards cutoff.\\
    \textbf{Keywords}: Markov processes, cutoff, $f$-divergence, R\'enyi divergence, reversibility, spectral gap, log-Sobolev constant\\
    \textbf{AMS 2020 subject classification}: 60J05, 60J10, 60J25, 60J27, 94A15, 94A17
\end{abstract}

\section{Introduction}

Given a family of Markov processes, the cutoff phenomenon describes the abrupt convergence to equilibrium of these processes when measured by a suitable probability metric. It was first observed in the context of card shuffling that entails the total variation (TV) cutoff \parencite{aldous1986shuffling, diaconis1996cutoff}. Since then, the cutoff phenomenon has been studied in a diverse suite of important models with different probability metrics, such as cutoff under separation distance of birth-death processes \parencite{diaconis2006Separation}, $L^p$ distances with $1 \leq p \leq \infty$ \parencite{chen2010l2, chen2006cutoff}, relative entropy or KL divergence for random walk on groups or samples of Markov chains \parencite{barrera2006cut, su1995methods}, squared Hellinger distance and cutoff on product chains \parencite{chen2018cutoffs}, cutoff for L\'evy driven OU processes under the Wasserstein distance \parencite{barrera2021cutoff}, cutoff for overdamped and underdamped Langevin dynamics \parencite{barrera2020thermalisation, lee2023asymptotic}, cutoff for the Dyson OU processes under KL, TV, squared Hellinger and Wasserstein distance \parencite{boursier2023universal}, cutoff for deep neural networks \parencite{avelin2022deep}, cutoff for the Ising model on lattice \parencite{lubetzky2013cutoff}, to name but a few.

Most of the methods to quantify convergence in above models are quite ad hoc, and they often require a detailed and complicated analysis of the specific processes, which makes them hard to replicate and apply on other models. Thus, people seek to find a unified criterion to check whether a family of Markov chains exhibits cutoff, without having to dealing with the complex details. In 2004, Yuval Peres proposed the famous ``Peres' conjecture" to check whether the underlying Markov chains exhibit cutoff via a simple product condition \parencite{peres2004american}: 
\begin{equation}\label{eq:oral Peres' conjecture}
    \textbf{spectral gap}\times \textbf{mixing time}\rightarrow\infty \Longleftrightarrow \textbf{cutoff exists},
\end{equation}
where the rigorous definition of the above equation will be explained later. For lazy reversible Markov chains, measured in \textbf{total variation distance}, the product condition is a necessary condition \parencite[Chapter 18.3]{levin2017markov}, and for quite many models, the condition is also sufficient, for example Glauber dynamics on Curie-Weiss model \parencite{ding2009mixing}, birth-and-death processes \parencite{ding2010total} and Markov chains on trees \parencite{basu2017characterization}. Counterexamples to sufficiency for TV-cutoff also exist, such as the Aldous' example and Pak's example summarized in \parencite[Section 6]{chen2008cutoff}. Beyond Peres' conjecture, \parencite{basu2017characterization, hermon2018technical} relate TV-cutoff to hitting times for reversible chains, and \parencite{salez2023cutoff} proposes the varentropy criterion for non-negatively curved Markov chains. In this article, we also deal with the general aspect of cutoff phenomenon, with cutoff under various information divergences as a special focus.

In information theory, an important and natural family of information divergences is known as the $f$-divergences \parencite{sason2016f}. With different choices of the function $f$, this family encompasses most of the common divergences in the literature such as the TV distance, KL divergence, squared Hellinger distance, $\chi^2$-divergence and $\alpha$-divergence $D_\alpha$. This family is also related to the R\'enyi divergences $R_\alpha$. We shall give a brief review of these divergences in Definition \ref{def:f-divergence} below. In the context of cutoff phenomenon, the total variation cutoff remains to be a focus in a majority of papers in the literature. It thus naturally raises a question: is total variation cutoff equivalent to other cutoff such as the separation cutoff? This has been answered in the negative in the paper \parencite{jonathan2016total}. In this paper, we aim at providing a systematic and unifying framework as well as a natural classification of these divergences under which the cutoff phenomnenon is equivalent within each type. Specifically, we propose four main types of information divergences in this context, namely $L^2$-type, $\mathrm{TV}$-type, separation-type and $\mathrm{KL}$ divergence. We prove that the cutoff phenomenon is equivalent among members within each type, along with some new product conditions to verify cutoff in these settings. We summarize these results in Table \ref{tab:merged}.

\begin{table}[htbp]
\centering
\begin{tabular}{|c|c|c|c|}
\hline
\multirow{2}{*}{} & \multirow{2}{*}{\textbf{Reversible}} & \multirow{2}{*}{\textbf{Normal}} & \textbf{Non-reversible} \\
 & & & \textbf{(bounded perturbation)}\\
\hline
\multirow{5}{*}{$L^2$-type}& $L^p \, (1<p\leq \infty)$ & \multicolumn{2}{c|}{$L^p \, (1<p<\infty)$} \\

 & \parencite{chen2008cutoff} &
 \multicolumn{2}{c|}{Theorem \ref{thm:characterization of worst-case L^p cutoff with 1<p<infty, normal}, \ref{thm:L^p cutoff, non-normal}}\\
\cline{2-4}

 & $R_\alpha \, (1<\alpha\leq \infty)$ & \multicolumn{2}{c|}{$R_\alpha \, (1<\alpha<\infty)$} \\

 & Theorem \ref{thm:Renyi cutoff alpha >= 2}, \ref{thm:alpha-divergence cutoff with 1<alpha<2} & \multicolumn{2}{c|}{Theorem \ref{thm:characterization of worst-case alpha/Renyi-divergence cutoff, non-reversible}}\\
\cline{2-4}

 & \multicolumn{3}{c|}{$D_\alpha \, (1<\alpha<\infty)$, Theorem \ref{thm:f divergence cutoff F_pq}, \ref{thm:alpha-divergence cutoff with 1<alpha<2}, \ref{thm:characterization of worst-case alpha/Renyi-divergence cutoff, non-reversible}} \\

\hline
\multirow{8}{*}{$\mathrm{TV}$-type} 
 & \multicolumn{3}{c|}{Total variation distance}\\
 
 & \multicolumn{3}{c|}{$D_\alpha \, (0<\alpha<1)$} \\
 
 & \multicolumn{3}{c|}{$R_\alpha \, (0<\alpha<1)$}\\

 & \multicolumn{3}{c|}{Squared Hellinger distance}\\

 & \multicolumn{3}{c|}{Vincze-Le Cam distance}\\

 & \multicolumn{3}{c|}{Jensen-Shannon divergence}\\

 & \multicolumn{3}{c|}{Bhattacharyya distance}\\

 & \multicolumn{3}{c|}{Theorem \ref{thm:equivalence between TV-type divergence cutoff and TV cutoff}, Corollary \ref{corollary:TV-type extension to Renyi divergence}}\\
\hline
$\mathrm{KL}$-type & \multicolumn{3}{c|}{$\mathrm{KL}$ divergence, Theorem \ref{thm:KL divergence cutoff via LSI}}\\
\hline
\multirow{3}{*}{Separation-type} 
 & \multicolumn{3}{c|}{Separation distance}\\
 
 & \multicolumn{3}{c|}{Reverse-$R_\infty$ divergence} \\

 & \multicolumn{3}{c|}{Section \ref{sec:separation-type divergences}}\\
\hline
\end{tabular}
\caption{Classification of some $f$-divergences and probability metrics by equivalence under cutoff phenomenon}
\label{tab:merged}
\end{table}

In order to justify that the above classification scheme is natural and is not due to artifacts in our proofs, we provide examples in which the family of Markov processes exhibit cutoff in one type but not in another, see the list below for pointers:
\begin{itemize}
    \item $L^2$-type and $\mathrm{TV}$-type are not equivalent: Aldous' example (Example \ref{ex:Aldous}), Pak's example (Example \ref{ex:Pak})
    \item $L^2$-type and $\mathrm{KL}$-type are not equivalent: Product chains (Example \ref{eg:product chains})
    \item $L^2$-type and separation-type are not equivalent: Pak's example (Example \ref{ex:Pak})
    \item $\mathrm{TV}$-type and $\mathrm{KL}$-type are not equivalent: Pak's example (Example \ref{ex:Pak})
    \item $\mathrm{TV}$-type and separation-type are not equivalent: \parencite{jonathan2016total}
    \item $\mathrm{KL}$-type and separation-type are not equivalent: Pak's example (Example \ref{ex:Pak})
\end{itemize}

We stress that, for possibly non-normal Markov generators that satisfy a bounded perturbation condition, we are able to characterize $L^p$-cutoff ($1< p < \infty$) with a product condition, and hence to prove its equivalence with $\alpha$-divergence or R\'enyi divergence cutoff in Section \ref{subsec:nonnormal}. It should be noted this is among the few results on cutoff for non-normal Markov processes. Owing to the absence of symmetry or reversibility, this direction has not received much attention in the literature.

The rest of this paper is organized as follows. In Section \ref{sec:prelim}, we provide a brief overview on various notions in Markov processes, cutoff phenomenon as well as information divergences. In Section \ref{sec:reversible}, we present some of our main results. Specifically, we first investigate the equivalence of $L^2$-type divergences among $\alpha$-divergence and R\'enyi divergence by introducing the so-called $\mathcal{F}_{p,q}$ family in Section \ref{sec:F_p,q family} and \ref{sec:alpha-divergence, 1<alpha<=2}, followed by studying $\pi$-weighted KL divergence and TV cutoff, and new product conditions in these settings in Section \ref{sec:pi-weighted KL and TV}. We then move on to discuss cutoff phenomenon of $\mathrm{TV}$-type divergences in Section \ref{sec:TV-type divergences} and separation-type in Section \ref{sec:separation-type divergences}. We conclude this section by illustrating the results with examples in Section \ref{subsec:examples}. We proceed to focus on cutoff phenomenon of normal Markov chains on finite state spaces in Section \ref{subsec:normal}, and then to non-normal Markov chains via perturbation theory in Section \ref{subsec:nonnormal}.

\subsection{Sketch of the proof}
The proof for the classification of equivalence relationships in Table \ref{tab:merged} relies on two observations. The first one is exponential contraction can imply cutoff, and the second point is comparison between mixing times, which is used to show the equivalence within each type. Notations can be found in Section \ref{sec:prelim}. In this subsection, under some specific divergence ``$\mathbf{dist}$", we denote $\mathbf{dist}^{(n)}(t)$ as the worst-case divergence between the 
$n^{th}$ process and its stationary distribution, and $t_{\mathrm{mix},n}(\mathbf{dist},\varepsilon)$ as its corresponding mixing time.

\textbf{Exponential contraction}: If there exists $\theta_n> 0$, $\theta_n t_{\mathrm{mix},n}(\mathbf{dist},\varepsilon)\rightarrow\infty$ for all $\varepsilon>0$, and two continuous and strictly increasing functions $\phi_1,\phi_2:[0,\infty)\rightarrow [0,\infty)$ with $\phi_1(0)=\phi_2(0)=0$, such that for any $u,v\in T$ and any $n$, 
\begin{equation*}
    \phi_1\left(\mathbf{dist}^{(n)}(u+v)\right)\leq e^{-\theta_n v}\phi_2\left(\mathbf{dist}^{(n)}(u)\right),
\end{equation*}
then cutoff under $\mathbf{dist}$ occurs at $t_{\text{mix},n}(\mathbf{dist},\varepsilon)$ with cutoff window $\theta_n^{-1}$.
\begin{itemize}
    \item The proof of this observation is via taking $u>t_{\text{mix},n}(\mathbf{dist},\varepsilon), v=\theta_n^{-1}c$ then letting $c\rightarrow +\infty$, and $u<t_{\text{mix},n}(\mathbf{dist},\varepsilon)+\theta_n^{-1}c, v=-\theta_n^{-1}c$ then letting $c\rightarrow -\infty$ respectively, according to Definition \ref{def:cutoff phenomenon}.

    \item If $\theta_n=\lambda_n$ the spectral gap, this proves sufficiency in Peres' conjecture without assumption of reversibility.

    \item We will use it to obtain the sufficient condition of cutoff under $\alpha$-divergence and R\'enyi divergence for $\alpha\in (1,\infty)$ in Section \ref{sec:reversible}.
\end{itemize}
This observation inspires us to study the contraction coefficients of processes under various divergences. Therefore, apart from spectral gap, other functional constants may be used to provide criterion for cutoff, for instance the log-Sobolev constant and modified log-Sobolev constant. It also relates cutoff phenomenon to the data processing constant in information theory, see \parencite{raginsky2016strong,makur2020comparison}. 

\textbf{Comparison between mixing times}: For two divergences $\mathbf{dist}_1,\mathbf{dist}_2$, suppose there exists $C_1,C_2>0$ and two continuous and strictly increasing functions $\psi,\psi:[0,\infty)\rightarrow [0,\infty)$ with $\psi(0)=\Psi(0)=0$, such that for all $\varepsilon>0$, 
\begin{equation*}
    C_1t_{\text{mix},n}(\mathbf{dist_2},\psi(\varepsilon))\leq t_{\text{mix},n}(\mathbf{dist_1},\varepsilon)\leq C_2t_{\text{mix},n}(\mathbf{dist_2},\Psi(\varepsilon)),
\end{equation*}
then under either of the following two situations, cutoff under $\mathbf{dist}_1$ and $\mathbf{dist}_2$ are equivalent:
\begin{enumerate}[label=(\roman*)]
    \item If Peres' conjecture holds for both $\mathbf{dist}_1$ and $\mathbf{dist}_2$;

    \item If $C_1=C_2=1$.
\end{enumerate}
The first situation holds for $L^2$-type divergences, and we use adaptations of Riesz-Thorin Interpolation Theorem in the proof. Remarkably in non-reversible cases, compared to the classical result of comparison between mixing times \parencite[Proposition 5.1]{chen2008cutoff} involving the mixing time of adjoint process which is hard to deal with, we make an extension to obtain a more practical result in Theorem \ref{thm:comparison between mixing times for L^p and R_alpha, non-reversible}. The second situation holds for divergences in TV-type and separation-type, and we use properties of $f$-divergence to prove.

\section{Preliminaries}\label{sec:prelim}
In this section, we will follow the discussions as in \parencite{chen2008cutoff} 
to introduce some basic definitions and properties related to the cutoff phenomenon for general Markov processes. We first begin by introducing some basic definitions of Markov processes.

\subsection{Markov process}
\label{sec:Markov process}
Consider a sequence of Markov processes $\{X_t^{(n)}, t\in T\}_{n=1}^\infty$ with $T$ being the time index set for the $n^{th}$ Markov process $\{X_t^{(n)}\}_{t\in T}$, where we may write $T=[0,\infty)$ for continuous-time Markov process, while $T=\mathbb N$ for discrete-time Markov chains. We denote $\mathcal{X}_n$ as the state space of $\{X_t^{(n)}\}_{t\in T}$, which can be continuous or discrete. Besides, we set $p_{n}(t,x,\cdot)$ with $x\in \mathcal{X}_n$ as the transition probability measure for $n^{th}$ Markov process. When analyzing some general properties of a Markov process without stressing the order of it in the sequence $n\geq 1$, for simplicity of notation, we may omit the subscript/superscript $n$ and simply use $p(t,x,\cdot)$, $\mathcal{X}$, $T$ and $\{X_t\}_{t\in T}$ to represent the transition probability measure, state space, time index set and Markov process respectively. 

For a Markov process $\{X_t\}_{t\in T}$ on state space $\mathcal{X}$ with $p(t,x,\cdot)$ as the transition probability measure, we define $P_t$ as the Markov semigroup, which satisfies 
\begin{equation*}
    P_t f(x):=\mathbb E^x[f(X_t)]=\int_\mathcal{X}f(y)p(t,x,dy), \quad x\in \mathcal{X}
\end{equation*}
for any bounded measurable function $f$ on $\mathcal{X}$, and it is easy to verify that $P_{t+s}=P_t\circ P_s$. Besides, we denote 
\begin{equation*}
    \mu P_t(A):=\int_{\mathcal{X}}p(t,x,A)\mu(dx)
\end{equation*}
as the probability measure of $X_t$ if the initial distribution $X_0\sim \mu$ for any probability measure $\mu$. Moreover, $\{X_t\}_{t\in T}$ has an infinitesimal generator $\mathcal{A}$, which satisfies $P_t=e^{t\mathcal{A}}$ if $T=[0,\infty)$ and $\mathcal{A}=P-I$ if $T=\mathbb N$. As to a sequence of Markov processes $\{X_t^{(n)}, t\in T\}_{n=1}^\infty$, we similarly define $P_{t,n}$ as the Markov semigroup of the $n^{th}$ process. 

Suppose a Markov process $\{X_t\}_{t\in T}$ on $\mathcal{X}$ admits  $\pi$ as its unique stationary distribution, we say the process is normal if the adjoint operator $P_t^*$ of $P_t$ on $L^2(\mathcal{X}, \pi)$ satisfies $P_tP_t^*=P_t^*P_t$, and the process is reversible if $P_t^*=P_t$. Particularly for a finite Markov chain with transition matrix $\left(P(x,y)\right)_{x,y\in \mathcal{X}}$, the chain is reversible if $\pi(x)P(x,y)=\pi(y)P(y,x)$ for all $x,y\in \mathcal{X}$, and the adjoint transition matrix $P^*$ of $P$ with respect to $\pi$ is given by the time reversal, i.e.
\begin{equation*}
    P^*(x,y)=\frac{\pi(y)P(y,x)}{\pi(x)}, \quad \forall x,y\in \mathcal{X}.
\end{equation*}

For any given initial distribution $\mu_0$, if we denote $h_t=\frac{d\mu_0P_t}{d\pi}$ as the probability density function of $\mu_0P_t$ with respect to $\pi$, we have $h_t=P_t^*h_0$, since for any $A\in \mathcal{B}(\mathcal{X})$, 
\begin{equation*}
    \int_A h_t d\pi =\int_A d\mu_0P_t=\mu_0P_t(A)=\int_\mathcal{X}P_t\mathbf{1}_{A}d\mu_0=\int_\mathcal{X}P_t\mathbf{1}_{A}h_0d\pi=\int_AP_t^*h_0d\pi.
\end{equation*}
A fundamental fact regarding this is the time evolution of $h_t$ under continuous-time, i.e.
\begin{equation}\label{eq:time volution of h_t}
    \frac{\partial}{\partial t}h_t=\mathcal{A}^*h_t, \quad t\in [0,\infty),
\end{equation}
where $\mathcal{A}^*$ is the adjoint operator of $\mathcal{A}$ on $L^2(\mathcal{X},\pi)$, and this equation can be also referred to as Kolmogorov's backward equation in the context of diffusion processes. 

For a finite Markov chain $\left\{X_k\right\}_{k=0}^\infty$ with transition matrix $P$ and stationary distribution $\pi$, we define its continuized chain $\{\widehat X_t\}_{t\geq 0}$ on the same finite state space $\mathcal{X}$ with transition matrix 
\begin{equation}\label{eq:continuized chain}
    P_t(x,y):=e^{t(P-I)}(x,y)=e^{-t}\sum_{k=0}^\infty \frac{P^k(x,y)t^k}{k!},
\end{equation}
then $\pi P=\pi$ is equivalent to $\pi P_t=\pi$, i.e. $\pi$ is also the stationary distribution of $\{\widehat X_t\}_{t\geq 0}$. Moreover, $P-I$ is also the generator of $\{\widehat X_t\}_{t\geq 0}$, which implies $P_t$ is the semigroup for the continuized chain.

Denote the Dirichlet form on $L^2(\mathcal{X},\pi)$ as $\mathcal{E}_{\mathcal{A}}(f,g):=\langle f, -\mathcal{A}g \rangle_{\pi}$. Particularly for a discrete-time Markov chain with generator $\mathcal{A}=P_1-I$ and one-step transition probability $p(x,\cdot)$ on state space $\mathcal{X}$, we have 
\begin{equation*}
    \mathcal{E}_{\mathcal{A}}(f,g)=\int_{x,y\in\mathcal{X}}f(x)\left(g(x)-g(y)\right)p(x,dy)\pi(dx),
\end{equation*}
and further if $\mathcal{A}$ is reversible,
\begin{align}
    \mathcal{E}_{\mathcal{A}}(f,g)&=\frac{1}{2}\int_{x,y\in\mathcal{X}}\left(f(x)-f(y)\right)\left(g(x)-g(y)\right)p(x,dy)\pi(dx),\label{eq:expression of Dirichelet f,g, reversible}\\
    \mathcal{E}_{\mathcal{A}}(f,f)&=\frac{1}{2}\int_{x,y\in\mathcal{X}}\left(f(x)-f(y)\right)^2p(x,dy)\pi(dx).\label{eq:expression of Dirichlet f,f}
\end{align}
When estimating the convergence performance of a Markov process $\{X_t\}_{t\in T}$ with generator $\mathcal{A}$ to stationary distribution $\pi$, spectral gap is an important tool, and the spectral gap is defined in terms of Dirichlet form.

\begin{definition}
    [Spectral gap of Markov process]
    \label{def:classical spectral gap}
    For a Markov process $\{X_t\}_{t\in T}$ with infinitesimal generator $\mathcal{A}$ and stationary distribution $\pi$, the spectral gap is defined as
    \begin{equation*}
        \lambda=\lambda(\mathcal{A}):=\inf \left\{\mathcal{E}_{\mathcal{A}}(f,f):f\in L^2(\mathcal{X},\pi), \mathbb E_\pi\left[f\right]=0, \|f\|_2=1\right\}.
    \end{equation*}
    If $\mathcal{A}$ is non-reversible, we have $\lambda(\mathcal{A})=\lambda(\mathcal{A}^*)=\lambda\left(\frac{\mathcal{A}+\mathcal{A}^*}{2}\right)$.
\end{definition}
Particularly for a finite Markov chain with transition matrix $P$, we can assume its generator as the generator of the continuized chain, i.e. $\mathcal{A}=P-I$, and its spectral gap satisfies $\lambda(\mathcal{A})=1-\lambda_1$ if $P$ is reversible, where $\lambda_1$ is the second largest eigenvalue of $P$. It can be readily seen that the discrete-time finite Markov chain and its continuized chain share the same spectral gap. A useful application of spectral gap is the following Proposition \ref{prop: convergence rate of Markov semigroup}, which involves a corollary of spectral mapping theorem, see \parencite[Chapter 4]{haase2018lectures} or \parencite{whitley1968spectral}.
\begin{proposition}
[Convergence rate of Markov semigroup]
\label{prop: convergence rate of Markov semigroup}
Assume a Markov process have a semigroup $P_t$, stationary distribution $\pi$ and spectral gap $\lambda\geq 0$. For all $f\in L^2(\mathcal{X},\pi)$ and $t\in T$, we have 
\begin{align}
    \|(P_t-\Pi)(f)\|_2 &\leq e^{-\lambda t}\|f\|_2, \quad T=[0,\infty), \label{eq:L^2 convergence rate, continuous time}\\
    \|(P_t-\Pi)(f)\|_2 &\leq \kappa^t\|f\|_2, \quad T=\mathbb N, \label{eq:L^2 convergence rate, discrete time}
\end{align}
where $\kappa$ is the second largest singular value of $P_1$. Moreover, if $P_t: L^2(\mathcal{X},\pi)\rightarrow L^2(\mathcal{X},\pi)$ is normal, we have
\begin{align*}
    \left\|P_t-\Pi\right\|_{L^2\rightarrow L^2}&=e^{-t\lambda}, \quad T=[0,\infty),\\
    \left\|P_t-\Pi\right\|_{L^2\rightarrow L^2}&=\kappa^t, \quad T=\mathbb N,
\end{align*}
where $\Pi f(x):=\pi(f), \forall x\in \mathcal{X}$.
\end{proposition}
This result applies in both cases of continuous and discrete-time, and will play an important role in analysis of cutoff phenomenon. Further explanations and pointers can be found in \parencite[Section 3.2]{chen2008cutoff}.

\subsection{Cutoff phenomenon and $f$-divergence}
Next, we will give a brief overview on cutoff phenomenon and recall some definitions and properties of $f$-divergences. The terminology \textbf{cutoff} describes a phenomenon that a sequence of Markov processes $\{X_t^{(n)},t\in T\}_{n=1}^\infty$ may exhibit a sharp transition in their mixing time to stationary distribution as $n\rightarrow \infty$ under suitable probability metrics or information divergences. Now, we give a formal definition of cutoff phenomenon from \parencite[Definition 2.1]{chen2008cutoff}. 

\begin{definition}
    [Cutoff phenomenon, \cite{chen2008cutoff}]
    \label{def:cutoff phenomenon}
    Consider a sequence of non-increasing functions $g_n:T\rightarrow [0,\infty]$ which vanish at infinity, i.e. $g_n(\infty)=0$ for all $n\geq 1$. If $M:=\limsup_{n\rightarrow \infty}g_n(0)>0$, where $M$ can be \textbf{infinity}, then
    \begin{enumerate}[label=(\roman*)]
         \item $\{g_n\}_{n=1}^\infty$ present a \textbf{precutoff} if there exist a sequence $\{t_n\}_{n=1}^\infty$ with $t_n>0$ and $b>a>0$ such that 
         \begin{equation*}
             \limsup_{n\rightarrow \infty} \sup_{t>bt_n} g_n(t)=0,\quad \liminf_{n\rightarrow \infty} \inf_{t<at_n} g_n(t)>0.
         \end{equation*}

         \item $\{g_n\}_{n=1}^\infty$ present a \textbf{cutoff} if there exists a sequence $\{t_n\}_{n=1}^\infty$ with $t_n>0$ such that for any $\varepsilon\in (0,1)$, 
         \begin{equation*}
             \limsup_{n\rightarrow \infty} \sup_{t>(1+\varepsilon)t_n} g_n(t)=0,
             \quad \liminf_{n\rightarrow \infty} \sup_{t<(1-\varepsilon)t_n} g_n(t)=M,
         \end{equation*}
         and in this case we say $\{g_n\}_{n=1}^\infty$ have a cutoff sequence $\{t_n\}_{n=1}^\infty$.

         \item $\{g_n\}_{n=1}^\infty$ present a $(t_n,w_n)$ cutoff if $t_n>0$, $w_n\geq 0$, $w_n=o(t_n)$, and 
         \begin{equation*}
             \lim_{c\rightarrow +\infty}\widetilde G(c)=0,\quad \lim_{c\rightarrow -\infty}\underline G(c)=M,
         \end{equation*}
         where 
         \begin{align*}
             \widetilde G(c):=\limsup_{n\rightarrow \infty} \sup_{t>t_n+cw_n} g_n(t), \quad \underline G(c):=\liminf_{n\rightarrow \infty} \sup_{t<t_n+cw_n} g_n(t),
         \end{align*}
         and in this case we call $\{w_n\}_{n=1}^\infty$ as the \textbf{cutoff window}.
    \end{enumerate}
\end{definition}

There is a deep connection between cutoff phenomenon and mixing times. Given a sequence of non-negative functions $g_n$ on $T$ as described earlier, for any $\varepsilon>0$, the mixing time of $g_n$ is defined as 
\begin{equation*}
    t(g_n,\varepsilon):=\inf \{t\in T: g_n(t)\leq \varepsilon\}.
\end{equation*}
Cutoff has an alternative characterization via mixing times:
\begin{proposition}
    [Cutoff and mixing time, \cite{chen2008cutoff}, Proposition 2.3]
    \label{prop:cutoff and mixing time}
    Consider a sequence of non-increasing functions $g_n:T\rightarrow [0,\infty]$ vanishing at infinity, and recall the mixing time defined above. Let $M:=\limsup_{n\rightarrow \infty}g_n(0)>0$, then the following statements hold.
    \begin{enumerate}[label=(\roman*)]
        \item $\{g_n\}_{n=1}^\infty$ has a precutoff if and only if there exists $\varepsilon>0$ and $C\geq 1$ such that 
        \begin{equation*}
            \limsup_{n\rightarrow\infty}\frac{t(g_n,\eta)}{t(g_n,\varepsilon)}\leq C, \quad \forall\hspace{0.1em} \eta\in (0,\varepsilon).
        \end{equation*}

        \item $\{g_n\}_{n=1}^\infty$ has a cutoff if and only if for all $0<\eta<\varepsilon<M$, 
        \begin{equation*}
            \lim_{n\rightarrow\infty}\frac{t(g_n,\eta)}{t(g_n,\varepsilon)}=1.
        \end{equation*}
    \end{enumerate}
\end{proposition}

Usually $g_n(t)$ can be chosen as a specific information divergence between the distribution of $X_t^{(n)}$ and the stationary distribution $\pi_n$ of the $n^{th}$ process. Typical examples include 
\begin{equation*}
        g_n(t)=\sup_{x\in \mathcal{X}}\mathrm{TV}(\delta_x P_{t,n},\pi_n),
\end{equation*}
where we have used the total variation distance to measure the distance between the distributions, see \parencite[Chapter 18]{levin2017markov} and \parencite{ding2010total}. Another common choice is the separation cutoff phenomenon in finite Markov chains with transition matrix $(P(x,y))_{x,y\in\mathcal{X}}$, which utilize
\begin{equation*}
    g_n(t)=\max_{x,y\in\mathcal{X}}\left\{1-\frac{P_{t,n}(x,y)}{\pi_n(y)}\right\},
\end{equation*}
see for example \parencite{diaconis2006Separation}. It turns out that this two choices entail $\mathrm{TV}$-type cutoff and separation-type cutoff, that we shall introduce in Section \ref{sec:reversible}.

One of the main aims of the manuscript is to study cutoff phenomenon under information-theoretic $f$-divergences. To this end, let us now recall its definition:

\begin{definition}
    [Csisz\'ar's $f$-divergence]\
    \label{def:f-divergence}
    Given two probability measures $\nu_1,\nu_2$ on $\mathcal{\mathcal{X}}$ with $\nu_1\ll \nu_2$, for a convex function $f:[0,\infty)\rightarrow \mathbb R$ such that $f(1)=0$, we define the $f$-divergence from $\nu_2$ to $\nu_1$ as
    \begin{equation*}
        D_f\left(\nu_1\|\nu_2\right):=\int_\mathcal{X}f\left(\frac{d\nu_1}{d\nu_2}\right)d\nu_2.
    \end{equation*}
\end{definition}

Many popular divergences belong to the family of $f$-divergences, and we refer to \parencite{sason2016f} and \parencite{van2014renyi} to give a few common examples:
\begin{itemize}
    \item Total variation (TV) distance: $f(t)=\dfrac{|t-1|}{2}$, denoted as $\mathrm{TV}(\nu_1,\nu_2)$.
    
    \item Relative entropy/Kullback-Leibler (KL) divergence: $f(t)=t\ln t-t+1$, denoted as $\mathrm{KL}(\nu_1\|\nu_2)$.

    \item $\chi^2$-divergence: $f(t)=|t-1|^2$, denoted as $\chi^2(\nu_1\|\nu_2)$.

    \item $\chi^p$-divergence ($p>0$): $f(t)=|t-1|^p$, denoted as $\chi^p(\nu_1\|\nu_2)$. When $p=1, 2$, we recover the total variation distance and $\chi^2$-divergence up to a constant.

    \item Jensen-Shannon divergence: $f(t)=t\ln t-(t+1)\ln \dfrac{t+1}{2}$, denoted as $\mathrm{JS}(\nu_1\|\nu_2)$. It also has the property 
    \begin{equation}\label{eq:JS divergence and KL divergence}
        \mathrm{JS}(\nu_1\|\nu_2)=\mathrm{KL}\left(\nu_1\bigg\|\frac{\nu_1+\nu_2}{2}\right)+\mathrm{KL}\left(\nu_2\bigg\|\frac{\nu_1+\nu_2}{2}\right).
    \end{equation}

    \item $\alpha$-divergence ($\alpha\in (0,1)\cup (1,\infty)$): $f(t)=f_\alpha(t)=\dfrac{t^\alpha-\alpha(t-1)-1}{\alpha-1}$, denoted as $D_{\alpha}(\nu_1\|\nu_2)$. A closely related divergence is the R\'enyi divergence defined as 
    \begin{equation}\label{eq:identity between alpha-divergence and Renyi divergence}
        R_\alpha(\nu_1\|\nu_2):=\frac{1}{\alpha-1}\ln (1+(\alpha-1)D_{\alpha}(\nu_1\|\nu_2))=\frac{1}{\alpha-1}\ln \int_\mathcal{X}\left(\frac{d\nu_1}{d\nu_2}\right)^\alpha d\nu_2.
    \end{equation}

    \item Squared Hellinger distance: $f(t)=\left(\sqrt{t}-1\right)^2$, denoted as $\mathrm{Hel}^2\left(\nu_1,\nu_2\right)$. 

    \item Vincze-Le Cam distance: $f(t)=\dfrac{(t-1)^2}{t+1}$, denoted as $\mathrm{LC}\left(\nu_1,\nu_2\right)$. An important relationship with $\chi^2$-divergence is that 
    \begin{equation}\label{eq:Le Cam distance and chi^2 divergence}
        \frac{1}{2}\mathrm{LC}\left(\nu_1,\nu_2\right)=\chi^2\left(\nu_1\bigg\|\frac{1}{2}\nu_1+\frac{1}{2}\nu_2\right)=\chi^2\left(\nu_2\bigg\|\frac{1}{2}\nu_1+\frac{1}{2}\nu_2\right).
    \end{equation}
\end{itemize}

In the following Proposition \ref{prop:properties of f-divergence}, we briefly recall some properties of information divergences in the literature:
\begin{proposition}
    [Some properties of information divergences]
    \label{prop:properties of f-divergence}
    Given two probability measures on $\mathcal{X}$ such that $\nu_1\ll \nu_2$, for a convex function $f:[0,\infty)\rightarrow \mathbb R$ such that $f(1)=0$, then the following properties hold.
    \begin{enumerate}[label=(\roman*)]
        \item\label{item:convex conjugate of f} \parencite[Theorem 5]{sason2016f} Denote $f^*(t):=tf\left(\frac{1}{t}\right)$ as the convex conjugate of $f(t)$, then we have 
        \begin{equation*}
            \sup_{\nu_1\neq \nu_2} \frac{D_f\left(\nu_1\|\nu_2\right)}{\mathrm{TV}(\nu_1,\nu_2)}=f(0)+f^*(0),
        \end{equation*}
        where $f^*(0):=\lim_{u\rightarrow \infty}\frac{f(u)}{u}$, and both $f(0)$ and $f^*(0)$ can be infinity.

        \item\label{item:monotonicity of Renyi divergence} (Monotonicity of $R_\alpha (\nu_1\|\nu_2)$ and $D_{\alpha}(\nu_1\|\nu_2)$ in $\alpha$)
        $R_\alpha (\nu_1\|\nu_2)$ and $D_\alpha(\nu_1\|\nu_2)$ are non-decreasing with respect to $\alpha\in (0,1)\cup (1,\infty)$, see \parencite[Theorem 3, 6]{van2014renyi}, \parencite[Theorem 36]{sason2016f} and \parencite{liese1987convex}. Moreover, we have
        \begin{equation*}
            \mathrm{KL}(\nu_1\|\nu_2)=\lim_{\alpha\nearrow 1}R_\alpha (\nu_1\|\nu_2),
        \end{equation*}
        and we can also write $\mathrm{KL}(\nu_1\|\nu_2)$ as $R_1(\nu_1\|\nu_2)$ to extend to the case of $\alpha=1$. We also have 
        \begin{equation*}
            \mathrm{KL}(\nu_1\|\nu_2)=\lim_{\alpha\nearrow 1}D_\alpha(\nu_1\|\nu_2),
        \end{equation*}
        where the limits in above two equations can be also taken from upperside if $D_\alpha(\nu_1\|\nu_2)<\infty$ for some $\alpha>1$. With the monoticity of R\'enyi divergence, we can also take the limit $\alpha\rightarrow\infty$ to define
        \begin{equation}\label{eq:R_infty}
            R_{\infty}(\nu_1\|\nu_2):=\lim_{\alpha\rightarrow\infty}R_{\alpha}(\nu_1\|\nu_2)=\ln \left(\mathop{\esssup}_{x\in \mathcal{X}}\frac{d\nu_1}{d\nu_2}\right).
        \end{equation}
        
        \item\label{item:Pinsker's inequality} (Pinsker's inequality, \cite{van2014renyi}, Theorem 31) For $\alpha\in (0,1]$, we have
        \begin{equation*}
            2\alpha \mathrm{TV}^2(\nu_1,\nu_2)\leq R_\alpha (\nu_1\|\nu_2).
        \end{equation*}
    \end{enumerate}
\end{proposition}

\subsection{$L^p$-cutoff}\label{sec:L^p cutoff}
The $f$-divergence family is a rich class of information divergences with elegant mathematical properties, and it naturally suggests that there are many potential choices for $g_n(t)$ to study cutoff phenomenon. One popular choice in the literature centers around the $L^p$-cutoff, which utilize the following divergence
\begin{equation}\label{eq:L^p divergence}
    d_p(x,t):=\left(\int_\mathcal{X}\left|\frac{d\delta_x P_t}{d\pi}-1\right|^p d\pi\right)^{\frac{1}{p}}=\|h(t,x,\cdot)-1\|_p, \quad p\geq 1,
\end{equation}
where $h(t,x,y)$ is the probability density function of $\delta_xP_t$ with respect to $\pi$. Taking supremum over $x\in \mathcal{X}$, we define
\begin{equation*}
    \overline d_p(t):=\sup_{x\in \mathcal{X}}d_p(x,t), \quad \widetilde d_p(t):=\pi\text{-}\mathop{\esssup}\limits_{x\in\mathcal{X}}d_p(x,t),
\end{equation*}
and take $g_n(t)=\overline d_{p,n}(t)$ or $g_n(t)=\widetilde d_{p,n}(t)$, where the $n$ in subscripts refer to the $n^{th}$ process. In particular when we take $p=1$, it recovers the total variation distance up to a constant. For the adjoint operator $P_t^*$ of $P_t$, we write
\begin{equation*}
        d_p^*(x,t):=\left(\int_\mathcal{X}\left|\frac{d\delta_x P_t^*}{d\pi}-1\right|^p d\pi\right)^{\frac{1}{p}},
\end{equation*}
and similarly
\begin{equation*}
    \overline d_p^*(t):=\sup_{x\in \mathcal{X}}d_p^*(x,t), \quad \widetilde d_p^*(t):=\pi\text{-}\mathop{\esssup}\limits_{x\in\mathcal{X}}d_p^*(x,t).
\end{equation*}

In most problems with mild conditions, the supremum and essential supremum defined above are the same, hence in this article we will focus on the latter one which we name as the ``worst-case" divergence. For $\varepsilon>0$, the worst-case $L^p$-mixing times are defined as
\begin{equation}
    \widetilde t_p(\varepsilon):=\inf\left\{t\in T:\widetilde d_p(t)\leq \varepsilon\right\}, \quad \widetilde t_p^*(\varepsilon):=\inf\left\{t\in T:\widetilde d_p^*(t)\leq \varepsilon\right\}\label{eq:definition of L^p mixing time}
\end{equation}

One of the reasons for the popularity of $L^p$-cutoff in the literature is that many useful tools can be applied regarding the space $L^p(\mathcal{X},\pi)$. We summarize some results from \parencite[Section 3.2, 3.3, 5.2, 5.3]{chen2008cutoff}, \parencite{dunford1988linear}, \parencite{stein2011functional} and \parencite{bernard2013interpolation} into Proposition \ref{prop:Important properties of L^p} and \ref{prop:Riesz-Thorin theorem}, which may be used in the rest of the paper. Here in the subscripts of norms, we use the shorthand $L^p$ to denote $L^p(\mathcal{X},\pi)$.
\begin{proposition}
    [Some properties of $L^p(\mathcal{X},\pi)$ and $d_p(x,t)$, \cite{chen2008cutoff}, \cite{dunford1988linear}]
    \label{prop:Important properties of L^p}
    Given a Markov process $\{X_t\}_{t\in T}$ with semigroup $P_t$, stationary distribution $\pi$ and $h(t,x,\cdot)$ defined before,
    for $p\in [1,\infty]$, let $q$ be the conjugate of $p$, i.e. $\frac{1}{p}+\frac{1}{q}=1$, then the following statements hold.
    \begin{enumerate}[label=(\roman*)]
        \item\label{item:L^p properties 1} Given a function $f\in L^p(\mathcal{X},\pi)$, its $L^p$-norm satisfies
        \begin{equation*}
            \|f\|_p=\sup \left\{\langle f,g\rangle_\pi: g\in L^q(\mathcal{X},\pi), \|g\|_q\leq 1\right\},
        \end{equation*}
        applied on $h(t,x,\cdot)$, we have
        \begin{equation*}
            d_p(x,t)=\sup \left\{(\delta_xP_t-\Pi)(g):g\in L^q(\mathcal{X},\pi), \|g\|_q \leq 1\right\}.
        \end{equation*}

        \item\label{item:L^p properties 2} The mapping $t\mapsto \widetilde d_p(t)$ is non-increasing and sub-multiplicative.

        \item\label{item:L^p properties 3} $\widetilde d_p(t)$ can be interpreted as operator norms, i.e.
        \begin{equation*}
            \widetilde d_p(t)=\left\|P_t-\Pi\right\|_{L^q\rightarrow L^\infty},
        \end{equation*}
        Similarly, we have
        \begin{equation*}
            \widetilde d_p^*(t)=\|P_t^*\|_{L^q\rightarrow L^\infty}=\|P_t\|_{L^1\rightarrow L^p}.
        \end{equation*}
    \end{enumerate}
\end{proposition}

\begin{proposition}
    [Riesz-Thorin Interpolation Theorem, \cite{stein2011functional}, \cite{bernard2013interpolation}]
    \label{prop:Riesz-Thorin theorem}
    
    Consider a linear operator $\mathcal{A}: L^{p_0}(\mathcal{X},\mu)\cup L^{p_1}(\mathcal{X},\mu)\rightarrow L^{q_0}(\mathcal{X},\nu)\cup L^{q_1}(\mathcal{X},\nu)$, where $p_0,p_1,q_0,q_1\in [1,\infty]$, and $\nu$ is semifinite. If there exists $M_0,M_1>0$ such that 
    \begin{align*}
        \|\mathcal{A}f\|_{q_0}&\leq M_0\|f\|_{p_0}, \quad \forall f\in L^{p_0}(\mathcal{X},\mu),\\
        \|\mathcal{A}f\|_{q_1}&\leq M_1\|f\|_{p_1}, \quad \forall f\in L^{p_1}(\mathcal{X},\mu),
    \end{align*}
    then for any $t\in (0,1)$ and $\dfrac{1}{p}=\dfrac{1-t}{p_0}+\dfrac{t}{p_1}$, $\dfrac{1}{q}=\dfrac{1-t}{q_0}+\dfrac{t}{q_1}$, we have 
    \begin{equation*}
        \|\mathcal{A}f\|_q\leq M_0^{1-t}M_1^t\|f\|_p, \quad \forall f\in L^p(\mathcal{X},\mu).
    \end{equation*}
\end{proposition}

The next result offers a characterization of $L^p$-cutoff for \textbf{reversible} Markov processes, which states that in such divergences, cutoff phenomenon occurs if and only if the associated product condition holds, that is, spectral gap multiplied by the $L^p$-mixing time tends to infinity.
\begin{proposition}
    [Characterization of $L^p$-cutoff, \cite{chen2008cutoff}, Theorem 5.3, 5.4]
    \label{prop:characterization of L^p cutoff}
    Consider a sequence of Markov processes $\{X_t^{(n)},t\in T\}_{n=1}^\infty$ with state space $\mathcal{X}_n$, stationary distribution $\pi_n$, spectral gap $\lambda_n\geq 0$, second largest singular value $0<\kappa_n\leq 1$ and semigroup $P_{t,n}$, where $P_{t,n}$ is \textbf{reversible} on $L^2(\mathcal{X}_n,\pi_n)$ for each $n\geq 1$. Let $g_n(t):=\widetilde d_{p,n}(t)$ and assume $\lim_{t\rightarrow \infty}g_n(t)=0$ for each $n$, if $T=[0,\infty)$, then the following statements are equivalent:
    \begin{enumerate}[label=(A\arabic*)]
        \item\label{item:A1} There exists some $p\in (1,\infty]$ and some $\varepsilon>0$ such that $\lambda_n\widetilde t_{p,n}(\varepsilon)$ tends to infinity.

        \item\label{item:A2} For any $p\in (1,\infty]$ and any $\varepsilon>0$, $\lambda_n\widetilde t_{p,n}(\varepsilon)$ tends to infinity.

        \item\label{item:A3} There exists some $p\in (1,\infty]$ such that precutoff occurs.

        \item\label{item:A4} For any $p\in (1,\infty]$, cutoff occurs.

        \item\label{item:A5} For any $p\in (1,\infty]$ and any $\varepsilon>0$, there is a $\left(\widetilde t_{p,n}(\varepsilon),\lambda_n^{-1}\right)$ cutoff.
    \end{enumerate}
    Here the $n$ in subscripts refer to the $n^{th}$ process. 
    
    Furthermore, if $T=\mathbb N$, assume for some $p\in (1,\infty]$ and $\varepsilon>0$, $\lim_{n\rightarrow \infty}\widetilde t_{p,n}(\varepsilon)=\infty$, and we substitute $\lambda_n'=\min \left\{1, -\ln \kappa_n\right\}$ into $\lambda_n$ in the items, then the items above are also equivalent. If we further assume $\lambda_n\rightarrow 0$ and that the Markov chains are lazy, i.e.
    \begin{equation}\label{eq:1/2 lazy chains}
        p_n(x,\{x\})\geq \frac{1}{2}, \quad \forall \hspace{0.1em} n\geq 1,\enspace x\in \mathcal{X}_n,
    \end{equation}
    where $p_n(x,\cdot)$ be the one-step transition probability of the $n^{th}$ chain. Then, we can also take $\lambda_n'=\min \left\{1, \lambda_n\right\}$.
\end{proposition}

\section{Reversible cases}\label{sec:reversible}
In this section, under a reversible setting, we extend Proposition \ref{prop:characterization of L^p cutoff} from $L^p$-mixing times to other mixing times induced by general $f$-divergences satisfying some mild conditions. Moreover, we uncover new relationships between cutoff under different divergences by relating their cutoff time and window, and develop a classification scheme among these divergences based on equivalence in characterization of cutoff phenomenon.

\subsection{$\mathcal{F}_{p,q}$ family and R\'enyi divergence with $\alpha\in [2,\infty]$}\label{sec:F_p,q family}

We begin this subsection by introducing a family of convex functions that we call the $\mathcal{F}_{p,q}$ family, which generates a few divergences, for instance the $\alpha$-divergence with $\alpha\in [2,\infty)$. The objective of this subsection is to prove that cutoff phenomenon are equivalent among members of the $\mathcal{F}_{p,q}$ family and to give a product condition for cutoff to occur. We then extend these results to R\'enyi divergence with $\alpha\in [2,\infty]$.

\begin{definition}
    [$\mathcal{F}_{p,q}$ family]
    Let $1<p\leq q<\infty$, we define
    \begin{align*}
        \mathcal{F}_{p,q}:=\bigg\{\textrm{convex }f : \mathbb{R}_+ \to \mathbb{R}, f(1) = 0:\enspace &\exists\, m,M>0 \enspace s.t.\enspace \forall x\in [0,\infty), \\
        &m\left(|x-1|^p+|x-1|^q\right)\leq f(x)\leq M\left(|x-1|^p+|x-1|^q\right)\bigg\}.
    \end{align*}
\end{definition}
\begin{example}\label{eg:alpha-divergence}
    For $\alpha$-divergence with $\alpha\in [2,\infty)$, the generator $f_\alpha (t)=\dfrac{t^\alpha-\alpha(t-1)-1}{\alpha-1}$ satisfies
    \begin{gather*}
        \lim_{t\rightarrow 1}\frac{f_\alpha(t)}{|t-1|^2+|t-1|^\alpha}=\frac{\alpha}{2}, \quad \lim_{t\rightarrow \infty}\frac{f_\alpha(t)}{|t-1|^2+|t-1|^\alpha}=\frac{1}{\alpha-1},\\
        \lim_{t\rightarrow 0}\frac{f_\alpha(t)}{|t-1|^2+|t-1|^\alpha}=\frac{1}{2},
    \end{gather*}
    which implies $f_\alpha\in \mathcal{F}_{2,\alpha}$. Another example is that if $f$ is strongly convex with $f(1)=f'(1)=0$ and $f''(t)$ is bounded on $[0,\infty)$, then $f\in \mathcal{F}_{2,2}$. However, for $1<\alpha<2$, the $\alpha$-divergence may \textbf{not} belong to any $\mathcal{F}_{p,q}$ family.
\end{example}

Analogous to the notations in Section \ref{sec:L^p cutoff}, for a Markov process $\{X_t\}_{t\in T}$ on state space $\mathcal{X}$ with semigroup $P_t$ and stationary distribution $\pi$, we define 
\begin{equation}
    d_f(x,t):=D_f(\delta_xP_t\|\pi),\quad
    \widetilde d_f(t):=\pi\text{-}\mathop{\esssup}\limits_{x\in\mathcal{X}}d_f(x,t),\label{eq:definition of worst-case f divergence}
\end{equation}
and the $f$-divergence mixing times 
\begin{equation}\label{eq:definition of worst-case f divergence mixing time}
    \widetilde t_f(\varepsilon):=\inf\left\{t\in T:\widetilde d_f(t)\leq \varepsilon\right\}.
\end{equation}
In the following result, for $f\in \mathcal{F}_{p,q}$, we give several equivalent criteria for the occurrence of $f$-divergence cutoff. Moreover, we will use $n$ in the subscripts to denote the $n^{th}$ process.

\begin{theorem}
    [Characterization of $f$-divergence cutoff for $\mathcal{F}_{p,q}, 1<p\leq q<\infty$]
    \label{thm:f divergence cutoff F_pq}
    Consider a sequence of Markov processes $\{X_t^{(n)},t\in T\}_{n=1}^\infty$ with state space $\mathcal{X}_n$, stationary distribution $\pi_n$, spectral gap $\lambda_n\geq 0$, second largest singular value $0<\kappa_n\leq 1$ and semigroup $P_{t,n}$, where $P_{t,n}$ is \textbf{reversible} on $L^2(\mathcal{X}_n,\pi_n)$ for each $n\geq 1$. Let $g_n(t):=\widetilde d_{f,n}(t)$ and assume $\lim_{t\rightarrow \infty}g_n(t)=0$ for each $n$. If $T=[0,\infty)$, then the following statements are equivalent: 
    \begin{enumerate}[label=(B\arabic*)]
        \item\label{item:B1} There exists some $1<p\leq q<\infty$, some $\varepsilon>0$ and some $f\in \mathcal{F}_{p,q}$ such that $\lambda_n\widetilde t_{f,n}(\varepsilon)$ tends to infinity.

        \item\label{item:B2} For any $1< p\leq q<\infty$, any $\varepsilon>0$ and any $f\in \mathcal{F}_{p,q}$, $\lambda_n\widetilde t_{f,n}(\varepsilon)$ tends to infinity.

        \item\label{item:B3} For any $1< p\leq q<\infty$ and any $f\in \mathcal{F}_{p,q}$, precutoff occurs.

        \item\label{item:B4} For any $1< p\leq q<\infty$ and any $f\in \mathcal{F}_{p,q}$, cutoff occurs.

        \item\label{item:B5} For any $1< p\leq q<\infty$, any $\varepsilon>0$ and any $f\in \mathcal{F}_{p,q}$, there is a $(\widetilde t_{f,n}(\varepsilon), \lambda_n^{-1})$ cutoff.
    \end{enumerate}
    Moreover, items \ref{item:B1} to \ref{item:B5} are all equivalent to items \ref{item:A1} to \ref{item:A5}. 
    
    For $T=\mathbb N$, assume for some $1< p\leq q<\infty$, some $\varepsilon>0$ and some $f\in \mathcal{F}_{p,q}$, $\lim_{n\rightarrow \infty}\widetilde t_{p,n}(\varepsilon)=\infty$. If we substitute $\lambda_n'=\min \{1, -\ln \kappa_n\}$ into $\lambda_n$ in the items, then the statements above also hold. Besides, if $\lambda_n\rightarrow 0$ and the chains are lazy, we can also take $\lambda_n'=\min \left\{1, \lambda_n\right\}$.
\end{theorem}
\begin{proof}
    We first consider the case of continuous-time. The proof sketch is that we will first prove items \ref{item:B2} to \ref{item:B5} are equivalent, then prove \ref{item:B2} to \ref{item:B5} and \ref{item:A1} to \ref{item:A5} are equivalent, and finally \ref{item:B1} and \ref{item:B2} to \ref{item:B5} are equivalent. 
    
    \ref{item:B2}$\Rightarrow$\ref{item:B5}: For any given $1< p\leq q<\infty$, $\varepsilon>0$ and $f\in \mathcal{F}_{p,q}$, 
    by definition we have for some $m,M>0$ depending on $f$ such that
    \begin{equation*}
        m\left(|x-1|^p+|x-1|^q\right)\leq f(x)\leq M\left(|x-1|^p+|x-1|^q\right),
    \end{equation*}
    which yields
    \begin{equation}\label{eq:relationship between d_f and d_p^p+d_q^q}
        m\left(d_{p,n}^p(x,t)+d_{q,n}^q(x,t)\right)\leq d_{f,n}(x,t)\leq M\left(d_{p,n}^p(x,t)+d_{q,n}^q(x,t)\right).
    \end{equation}
    Following the proof in \parencite[Theorem 3.3]{chen2008cutoff}, we denote $\mu_{t,n}^x=\delta_{x,n} P_{t,n}$ and let $\frac{1}{p}+\frac{1}{p'}=1$, $\frac{1}{q}+\frac{1}{q'}=1$. For $t=u+v$ and any $g\in L^{p'}(\mathcal{X}_n,\pi_n)$, we have
    \begin{equation}\label{eq:Kolmogorov Chapman}
        \left(\mu_{t,n}^x-\pi_n\right)(g)=\left(\mu_{u,n}^x-\pi_n\right)\left(P_{v,n}-\Pi_n\right)(g),
    \end{equation}
    then by H\"older's inequality and Riesz-Thorin Interpolation Theorem as in Proposition \ref{prop:Riesz-Thorin theorem}, we have
    \begin{align}
        \left|\left(\mu_{t,n}^x-\pi_n\right)(g)\right|&\leq d_{p,n}(x,u)\left\|(P_{v,n}-\Pi_n)(g)\right\|_{p'}\nonumber\\
        &\leq d_{p,n}(x,u)2^{|1-2/p|}e^{-v\lambda_n (1-|1-2/p|)}\|g\|_{p'},\label{eq:convergence rate of d_p(x,u)}
    \end{align}
    where the second inequality comes from $\left\|P_{v,n}-\Pi_n\right\|_{L^1\rightarrow L^1}\leq 2$, $\left\|P_{v,n}-\Pi_n\right\|_{L^\infty\rightarrow L^\infty}\leq 2$ and \eqref{eq:L^2 convergence rate, continuous time} in Proposition \ref{prop: convergence rate of Markov semigroup}. Taking supremum over $g\in L^{p'}(\mathcal{X}_n,\pi_n), \|g\|_{p'}=1$ and $g\in L^{q'}(\mathcal{X}_n,\pi_n), \|g\|_{q'}=1$ respectively, according to Proposition \ref{prop:Important properties of L^p} item \ref{item:L^p properties 1}, we have
    \begin{align*}
        d_{p,n}^p(x,u+v)&\leq d_{p,n}^p(x,u)2^{|p-2|}e^{-v\lambda_n(p-|p-2|)},\\
        d_{q,n}^q(x,u+v)&\leq d_{q,n}^q(x,u)2^{|q-2|}e^{-v\lambda_n(q-|q-2|)},
    \end{align*}
    Plugging into \eqref{eq:relationship between d_f and d_p^p+d_q^q}, we have
    \begin{align*}
        d_{f,n}(x,u+v)&\leq M C_{p,q}\left(d_{p,n}^p(x,u)+d_{q,n}^q(x,u)\right)e^{-v\lambda_n a_{p,q}}\\
        &\leq \frac{MC_{p,q}}{m}\cdot d_{f,n}(x,u)\cdot e^{-v\lambda_n a_{p,q}},
    \end{align*}
    where $C_{p,q}:=\max \left\{2^{|p-2|},2^{|q-2|}\right\}>0$, $a_{p,q}:=\min \left\{p-|p-2|,q-|q-2|\right\}>0$. Taking supremum over $x\in \mathcal{X}$, we have 
    \begin{equation}\label{eq:exponential convergence of d_f}
        \widetilde d_{f,n}(u+v)\leq \frac{MC_{p,q}}{m}\cdot d_{f,n}(u)\cdot e^{-v\lambda_n a_{p,q}}.
    \end{equation}
    Now taking $u> \widetilde t_{f,n}(\varepsilon)$, $v=\lambda_n^{-1}c$ with $c>0$ in \eqref{eq:exponential convergence of d_f}, by monotonicity of $\widetilde d_{f,n}(t)$ in $t$ as shown in Proposition \ref{prop:Important properties of L^p} item \ref{item:L^p properties 2}, we have
    \begin{equation*}
        \widetilde G(c)=\limsup_{n\rightarrow\infty}\sup_{t>\widetilde t_{f,n}(\varepsilon)+c\lambda_n^{-1}} \widetilde d_{f,n}(t)\leq \frac{MC_{p,q}}{m}\cdot \varepsilon e^{-c a_{p,q}},
    \end{equation*}
    and similarly taking $0<u<\widetilde t_{f,n}(\varepsilon)+\lambda_n^{-1}c$, $v=-\lambda_n^{-1}c$ with $c<0$, we have
    \begin{equation*}
        \underline G(c)=\liminf_{n\rightarrow\infty}\inf_{t<\widetilde t_{f,n}(\varepsilon)+c\lambda_n^{-1}} \widetilde d_{f,n}(t)\geq \frac{m}{MC_{p,q}}\cdot \varepsilon e^{-ca_{p,q}}.
    \end{equation*}
    The desired result follows by taking $c\rightarrow +\infty$ and $c\rightarrow -\infty$ respectively.
    
    \ref{item:B5}$\Rightarrow$\ref{item:B4}$\Rightarrow$\ref{item:B3}: By definition.

    \ref{item:B3}$\Rightarrow$\ref{item:B2}: We follow the proof in \parencite[Theorem 4.2]{chen2008cutoff}. According to \eqref{eq:relationship between d_f and d_p^p+d_q^q} and Proposition \ref{prop:Important properties of L^p} item \ref{item:L^p properties 3}, we have 
    \begin{align*}
        \widetilde d_{f,n}(t)&\geq m\widetilde d_{q,n}^q(t)=m\left\|P_{t,n}-\Pi_n\right\|_{L^{q'}\rightarrow \infty}\\
        &\geq m\left\|P_{t,n}-\Pi_n\right\|_{L^{q'}\rightarrow L^{q'}},
    \end{align*}
    where $\frac{1}{q}+\frac{1}{q'}=1$. By Riesz-Thorin Interpolation Theorem, we have 
    \begin{align*}
        e^{-\lambda_n t}=\left\|P_{t,n}-\Pi_n\right\|_{L^2\rightarrow L^2}&\leq \left\|P_{t,n}-\Pi_n\right\|_{L^{q'}\rightarrow L^{q'}}^{q'/2}\left\|P_{t,n}-\Pi_n\right\|_{L^\infty\rightarrow L^\infty}^{1-q'/2}\\
        &\leq 2^{1-q'/2}\left\|P_{t,n}-\Pi_n\right\|_{L^{q'}\rightarrow L^{q'}}^{q'/2}, \quad q'\in (1,2],
    \end{align*}
    and 
    \begin{align*}
        e^{-\lambda_n t}=\left\|P_{t,n}-\Pi_n\right\|_{L^2\rightarrow L^2}&\leq \left\|P_{t,n}-\Pi_n\right\|_{L^1\rightarrow L^1}^{1-\frac{q'}{2(q'-1)}}\left\|P_{t,n}-\Pi_n\right\|_{L^{q'}\rightarrow L^{q'}}^{\frac{q'}{2(q'-1)}}\\
        &\leq 2^{1-\frac{q'}{2(q'-1)}}\left\|P_{t,n}-\Pi_n\right\|_{L^{q'}\rightarrow L^{q'}}^{\frac{q'}{2(q'-1)}}, \quad q'\in (2,\infty),
    \end{align*}
    hence we have 
    \begin{align}
        \left\|P_{t,n}-\Pi_n\right\|_{L^{q'}\rightarrow L^{q'}}&\geq 2^{1-2/q'}e^{-2\lambda_nt/q'}\geq \frac{1}{2}e^{-2\lambda_nt},\quad q'\in (1,2],\label{eq:lower bound of q' to q' with 1<q'<=2}\\
        \left\|P_{t,n}-\Pi_n\right\|_{L^{q'}\rightarrow L^{q'}}&\geq 2^{-(q'-2)/q'}e^{-2\lambda_nt (q'-1)/q'}\geq \frac{1}{2}e^{-2\lambda_nt}, \quad q'\in (2,\infty),\label{eq:lower bound of q' to q' with q'>2}
    \end{align}
    which implies 
    \begin{equation}\label{eq:lower bound of d_f}
        \widetilde d_{f,n}(t)\geq \frac{m}{2}e^{-2\lambda_nt}.
    \end{equation}
    Next, we suppose there is a precutoff sequence $\{s_n\}_{n=1}^\infty$, then there exist $0<a<b$ and $\delta>0$ such that 
    \begin{gather}
        2\delta=\liminf_{n\rightarrow\infty}\widetilde d_{f,n}(as_n)>0,\label{eq:d_f as_n}\\
        0=\limsup_{n\rightarrow\infty}\widetilde d_{f,n}(bs_n)\geq \frac{m}{2}\limsup_{n\rightarrow\infty} e^{-2b\lambda_ns_n},\label{eq:d_f bs_n}
    \end{gather}
    here \eqref{eq:d_f as_n} implies $s_n=\mathcal{O}(\widetilde t_{f,n}(\delta))$, otherwise for some small $\eta>0$, 
    \begin{equation*}
        \widetilde d_{f,n}(as_n)=\widetilde d_{f,n}\left(\frac{as_n}{(1+\eta)\widetilde t_{f,n}(\delta)}(1+\eta)\widetilde t_{f,n}(\delta)\right)\leq \delta, \quad as\enspace n\rightarrow\infty.
    \end{equation*}
    Combined with \eqref{eq:d_f bs_n} which indicates $\lambda_ns_n\rightarrow \infty$, we have $\lambda_n\widetilde t_{f,n}(\delta)\rightarrow\infty$. Similar to the proof of \ref{item:B2}$\Rightarrow$\ref{item:B5} where only some fixed $p,q,f$ and $\varepsilon$ are studied, it is easy to verify that $\{\widetilde t_{f,n}(\delta)\}_{n=1}^\infty$ is a cutoff sequence. Further by \parencite[Corollary 2.5]{chen2008cutoff}, $\{\widetilde t_{f,n}(\varepsilon)\}_{n=1}^\infty$ is a cutoff sequence with $\widetilde t_{f,n}(\varepsilon)\sim \widetilde t_{f,n}(\delta)$ for any $\varepsilon>0$, hence $\lambda_n\widetilde t_{f,n}(\varepsilon)\rightarrow\infty$ for any $\varepsilon>0$.

    \ref{item:B2} to \ref{item:B5}$\iff$\ref{item:A1} to \ref{item:A5}: It suffices to prove \ref{item:B2}$\Rightarrow$\ref{item:A1} and \ref{item:A2}$\Rightarrow$\ref{item:B2}. For any $p>1$, $f$ is convex with $f(1)=0$, and any $\varepsilon>0$, we denote 
    \begin{equation*}
        \widetilde T_{p,n}(\varepsilon):= \left\{t\in T:\widetilde d_{p,n}(t)\leq \varepsilon\right\}, \quad \widetilde T_{f,n}(\varepsilon):= \left\{t\in T:\widetilde d_{f,n}(t)\leq \varepsilon\right\}.
    \end{equation*}
    Next, for any given $1< p\leq q<\infty$ and $f\in \mathcal{F}_{p,q}$, by \eqref{eq:relationship between d_f and d_p^p+d_q^q}, if $t\in \widetilde T_{q,n}(\varepsilon)$, then
    \begin{equation*}
        \widetilde d_{f,n}(t)\leq M(\varepsilon^p+\varepsilon^q), 
    \end{equation*}
    where we have used monotonicity of $L^p$ distance in $p$. Similarly if $t\in \widetilde T_{f,n}(\varepsilon)$, we have 
    \begin{equation*}
        \widetilde d_{q,n}(t)\leq \left(\frac{\varepsilon}{m}\right)^{\frac{1}{q}},
    \end{equation*}
    and these two inequalities above imply
    \begin{equation*}
        \widetilde T_{q,n}(\varepsilon)\subset \widetilde T_{f,n}\left(M(\varepsilon^p+\varepsilon^q)\right), \quad \widetilde T_{f,n}(\varepsilon)\subset \widetilde T_{q,n}\left(\left(\varepsilon/m\right)^{\frac{1}{q}}\right),
    \end{equation*}
    taking infimum we obtain 
    \begin{equation}\label{eq:relationship between mixing times of t_f and t_q}
        \widetilde t_{q,n}(\varepsilon)\geq \widetilde t_{f,n}\left(M(\varepsilon^p+\varepsilon^q)\right), \quad \widetilde t_{f,n}(\varepsilon)\geq \widetilde t_{q,n}\left(\left(\varepsilon/m\right)^\frac{1}{q}\right).
    \end{equation}
    Now if \ref{item:B2} holds for some $1<p\leq q<\infty$ and 
    $\varepsilon'=M(\varepsilon^p+\varepsilon^q)$ with some $f\in \mathcal{F}_{p,q}$ such that $\lambda_n \widetilde t_{f,n}(\varepsilon')\rightarrow \infty$, by the first inequality in \eqref{eq:relationship between mixing times of t_f and t_q}, we have $\lambda_n \widetilde t_{q,n}(\varepsilon)\rightarrow \infty$, which is \ref{item:A1}. Moreover, if \ref{item:A2} holds, then for any given $1<q<\infty$ and any $\varepsilon''=\left(\varepsilon/m\right)^\frac{1}{q}$ such that $\lambda_n \widetilde t_{q,n}(\varepsilon'')\rightarrow \infty$, by the second inequality in \eqref{eq:relationship between mixing times of t_f and t_q} we have $\lambda_n\widetilde t_{f,n}(\varepsilon)\rightarrow \infty$, which is \ref{item:B2}.

    \ref{item:B1}$\iff$\ref{item:B2} to \ref{item:B5}: We only need to prove \ref{item:B1}$\Rightarrow$\ref{item:A1}, then by \ref{item:A1}$\Rightarrow$\ref{item:B2} to \ref{item:B5} and \ref{item:B2} to \ref{item:B5}$\Rightarrow$\ref{item:B1} we can get the result. Suppose there exist some $1<p\leq q<\infty$, some $\varepsilon>0$ and some $f\in \mathcal{F}_{p,q}$ such that $\lambda_n\widetilde t_{f,n}(\varepsilon)\rightarrow\infty$. Similar to the proof of \ref{item:B2}$\Rightarrow$\ref{item:B5}, there is a $\{\widetilde t_{f,n}(\varepsilon),\lambda_n^{-1}\}_{n=1}^\infty$ cutoff. According to \parencite[Corollary 2.5]{chen2008cutoff}, for any $\delta>0$, $\{\widetilde t_{f,n}(\delta)\}_{n=1}^\infty$ is a cutoff sequence. Again similar to the proof of \ref{item:B3}$\Rightarrow$\ref{item:B2} where only some fixed $p,q,f$ are studied, we can obtain that $\lambda_n\widetilde t_{f,n}(\delta)\rightarrow\infty$ for any $\delta>0$. Then by the first inequality in \eqref{eq:relationship between mixing times of t_f and t_q}, for any $\delta>0$, there exists $\delta_0>0$ which satisfies $\delta=M(\delta_0^p+\delta_0^q)$ such that $\lambda_n\widetilde t_{q,n}(\delta_0)\rightarrow\infty$, and this yields \ref{item:A1}.

    As to the case of $T=\mathbb N$, the proof is similar.
\end{proof}

As shown in Example \ref{eg:alpha-divergence}, $\alpha$-divergence belongs to the $\mathcal{F}_{p,q}$ family for $\alpha\in [2,\infty)$, and we note that R\'enyi divergence is a monotonic function of $\alpha$-divergence, that is,
\begin{equation*}
    R_\alpha(\delta_xP_t\|\pi)=\frac{1}{\alpha-1}\ln \left(1+(\alpha-1)D_\alpha(\delta_xP_t\|\pi)\right).
\end{equation*}
In view of the above, in the following result we shall give equivalent conditions for cutoff phenomenon under R\'enyi divergence for $\alpha\in [2,\infty]$, where we recall the $R_\infty$ divergence defined in \eqref{eq:R_infty}. Analogous to the notations introduced earlier, we denote
\begin{gather*}
    d_{f_\alpha}(x,t):=D_{\alpha}(\delta_xP_t\|\pi),\quad d_{R_\alpha}(x,t):=R_{\alpha}(\delta_xP_t\|\pi),\\
    \widetilde d_{f_\alpha}(t):=\pi\text{-}\mathop{\esssup}\limits_{x\in\mathcal{X}}d_{f_\alpha}(x,t),\quad \widetilde d_{R_\alpha}(t):=\pi\text{-}\mathop{\esssup}\limits_{x\in\mathcal{X}}d_{R_\alpha}(x,t),
\end{gather*}
and the mixing times with respect to R\'enyi divergence for $\varepsilon>0$ are defined as
\begin{equation*}
    \widetilde t_{f_\alpha}(\varepsilon):=\inf\left\{t\in T:\widetilde d_{f_\alpha}(t)\leq \varepsilon\right\},\quad \widetilde t_{R_\alpha}(\varepsilon):=\inf\left\{t\in T:\widetilde d_{R_\alpha}(t)\leq \varepsilon\right\},
\end{equation*}
and we still use $n$ in the subscripts to denote the $n^{th}$ process.
\begin{theorem}
    [Characterization of R\'enyi divergence cutoff for $2\leq \alpha\leq \infty$]
    \label{thm:Renyi cutoff alpha >= 2}
    Consider a sequence of Markov processes $\{X_t^{(n)},t\in T\}_{n=1}^\infty$ with state space $\mathcal{X}_n$, stationary distribution $\pi_n$, spectral gap $\lambda_n\geq 0$, second largest singular value $0<\kappa_n\leq 1$ and semigroup $P_{t,n}$, where $P_{t,n}$ is \textbf{reversible} on $L^2(\mathcal{X}_n,\pi_n)$ for each $n\geq 1$. Let $g_n(t):=\widetilde d_{R_{\alpha},n}(t)$, and assume $\lim_{t\rightarrow \infty}g_n(t)=0$ for each $n$. If $T=[0,\infty)$, then the following statements are equivalent: 
    \begin{enumerate}[label=(C\arabic*)]
        \item\label{item:C1} There exists some $\alpha\in [2,\infty]$ and some $\varepsilon>0$ such that $\lambda_n\widetilde t_{R_\alpha,n}(\varepsilon)\rightarrow\infty$.

        \item\label{item:C2} For any $\alpha\in [2,\infty]$ and any $\varepsilon>0$, $\lambda_n\widetilde t_{R_\alpha,n}(\varepsilon)\rightarrow\infty$.

        \item\label{item:C3} For any $\alpha\in [2,\infty]$ and any $\varepsilon>0$, precutoff occurs.

        \item\label{item:C4} For any $\alpha\in [2,\infty]$ and any $\varepsilon>0$, cutoff occurs.

        \item\label{item:C5} For any $\alpha\in [2,\infty]$ and any $\varepsilon>0$, there is a $(\widetilde t_{R_\alpha,n}(\varepsilon), \lambda_n^{-1})$ cutoff.
    \end{enumerate}
    Moreover, items \ref{item:C1} to \ref{item:C5} are equivalent to items \ref{item:B1} to \ref{item:B5} and \ref{item:A1} to \ref{item:A5}. 
    
    For $T=\mathbb N$, assume for some $\alpha\in [2,\infty]$ and some $\varepsilon>0$, $\lim_{n\rightarrow \infty}\widetilde t_{R_\alpha,n}(\varepsilon)=\infty$. If we substitute $\lambda_n'=\min \{1, -\ln \kappa_n\}$ into $\lambda_n$ in the items, then the statements above also hold. Besides, if $\lambda_n\rightarrow 0$ and the chains are lazy, we can also take $\lambda_n'=\min \left\{1, \lambda_n\right\}$.
\end{theorem}

\begin{remark}
    Theorem \ref{thm:f divergence cutoff F_pq} and \ref{thm:Renyi cutoff alpha >= 2} indicate that for any given sequence of Markov processes, $L^p$-cutoff with $p\in (1,\infty]$, $\alpha$-divergence cutoff with $\alpha\in [2,\infty)$ and R\'enyi divergence cutoff with $\alpha\in [2,\infty]$ are all equivalent, and we call these three types of divergences as well as members of the $\mathcal{F}_{p,q}$ family with $1<p\leq q<\infty$ as \boldmath$L^2$\textbf{-type} divergence under cutoff phenomenon.
\end{remark}

\begin{proof}
    We only consider the case of continuous-time, and the proof for $T=\mathbb N$ is similar. An outline of the proof is that we first consider the situation of $\alpha\in [2,\infty)$, under which we prove \ref{item:C2} to \ref{item:C5} are equivalent, then prove \ref{item:C1}$\Rightarrow$\ref{item:B1} and \ref{item:B2}$\Rightarrow$\ref{item:C2}$\Rightarrow$\ref{item:C1}. Next, we incorporate the case of $\alpha=\infty$.

    Under $\alpha\in [2,\infty)$:

    \ref{item:C2}$\Rightarrow$\ref{item:C5}: For any given $\alpha\in [2,\infty)$ and $\varepsilon>0$, $f_\alpha\in \mathcal{F}_{2,\alpha}$ implies we can substitute $p=2$ and $q=\alpha$ into \eqref{eq:relationship between d_f and d_p^p+d_q^q}, and therefore by \eqref{eq:exponential convergence of d_f} we have
    \begin{equation*}
        \widetilde d_{f_\alpha,n}(u+v)\leq \frac{M2^{\alpha-2}}{m}\cdot d_{f_\alpha,n}(u)\cdot e^{-2v\lambda_n },
    \end{equation*}
    which yields
    \begin{align*}
        \frac{\widetilde d_{R_\alpha,n}(u+v)}{\widetilde d_{R_\alpha,n}(u)}&=\frac{\ln \left(1+(\alpha-1)\widetilde d_{f_\alpha,n}(u+v)\right)}{\ln \left(1+(\alpha-1)\widetilde d_{f_\alpha,n}(u)\right)}\\
        &\leq \frac{\widetilde d_{f_\alpha,n}(u+v)}{\widetilde d_{f_\alpha,n}(u)}\cdot \left(1+(\alpha-1)\widetilde d_{f_\alpha,n}(u)\right)\\
        &\leq \frac{M2^{\alpha-2}}{m}\cdot e^{-2v\lambda_n}\cdot \exp \left((\alpha-1)\widetilde d_{R_\alpha,n}(u)\right).
    \end{align*}
    Let $\varphi(t)=te^{(\alpha-1)t}$ be an increasing function with respect to $t\in (0,\infty)$, we have 
    \begin{equation*}
        \widetilde d_{R_\alpha,n}(u+v)\leq \frac{M2^{\alpha-2}}{m}\cdot e^{-2v\lambda_n}\cdot \varphi\left(\widetilde d_{R_\alpha,n}(u)\right),
    \end{equation*}
    using the same argument as the proof in \ref{item:B2}$\Rightarrow$\ref{item:B5}, we take $u>\widetilde t_{R_\alpha,n}(\varepsilon), v=\lambda_n^{-1}c$ and $0<u<\widetilde t_{R_\alpha,n}(\varepsilon)-\lambda_n^{-1}c, v=-\lambda_n^{-1}c$ with $c>0$, as $\varphi: \mathbb R^+\rightarrow \mathbb R^+$ is strictly increasing, we get the result.

    \ref{item:C5}$\Rightarrow$\ref{item:C4}$\Rightarrow$\ref{item:C3}: By definition.

    \ref{item:C3}$\Rightarrow$\ref{item:C2}: By \eqref{eq:lower bound of d_f}, we have 
    \begin{equation*}
        \widetilde d_{R_\alpha,n}(t)\geq \frac{1}{\alpha-1}\ln \left(1+\frac{m(\alpha-1)}{2}e^{-2\lambda_nt}\right),
    \end{equation*}
    which yields the result via a similar argument in the proof of \ref{item:B3}$\Rightarrow$\ref{item:B2}.

    \ref{item:C1}$\Rightarrow$\ref{item:B1}: Suppose for some $\alpha\in [2,\infty)$ and some $\varepsilon>0$, $\lambda_n\widetilde t_{R_\alpha,n}(\varepsilon)\rightarrow\infty$. Similar to the proof of \ref{item:B2} to \ref{item:B5}$\iff$\ref{item:A1} to \ref{item:A5}, for any $\varepsilon'>0$, it is easy to verify that 
    \begin{equation}\label{eq:relationship between mixing time of D_alpha and R_alpha}
        \widetilde t_{f_\alpha,n}(\varepsilon')= \widetilde t_{R_\alpha,n}\left(\frac{1}{\alpha-1}\ln (1+(\alpha-1)\varepsilon')\right),
    \end{equation}
    then if take some $\varepsilon'$ such that $\varepsilon=\frac{1}{\alpha-1}\ln (1+(\alpha-1)\varepsilon')$, $\lambda_n \widetilde t_{f_\alpha,n}(\varepsilon')\rightarrow\infty$.

    \ref{item:B2}$\Rightarrow$\ref{item:C2}$\Rightarrow$\ref{item:C1}: Take $p=2$, $q=\alpha$ and $f_\alpha\in \mathcal{F}_{2,\alpha}$, for any $\varepsilon>0$ we have $\lambda_n\widetilde d_{f_\alpha,n}(\varepsilon)\rightarrow\infty$, then by \eqref{eq:relationship between mixing time of D_alpha and R_alpha} we get the result.

    Under $\alpha\in [2,\infty]$:

    For continuous-time setting, we show that 
    \begin{equation}\label{eq:R_infty and R_2, reversible}
        \widetilde d_{R_\infty,n}(2t)=\widetilde d_{R_2,n}(t).
    \end{equation}
    We first recall that for any $f\in L^1(\mathcal{X}_n,\pi_n)$,
    \begin{equation*}
        \left\|P_{2t,n}f\right\|_{\infty}=\sup_{\|g\|_1\leq 1}\langle P_{2t,n}f, g\rangle_{\pi_n},
    \end{equation*}
    hence by reversibility,
    \begin{align*}
        \left\|P_{2t,n}\right\|_{L^1\rightarrow L^\infty}&=\sup_{\|f\|_1\leq 1, \|g\|_1\leq 1}\langle P_{2t,n}f, g\rangle_{\pi_n}\geq \sup_{\|f\|_1\leq 1}\langle P_{2t,n}f, f\rangle_{\pi_n}\\
        &=\sup_{\|f\|_1\leq 1}\langle P_{t,n}f, P_{t,n}^* f\rangle_{\pi_n}=\sup_{\|f\|_1\leq 1}\left\|P_{t,n}f\right\|_2^2\\
        &=\left\|P_{t,n}\right\|_{L^2\rightarrow L^\infty}^2.
    \end{align*}
    Moreover, we already have the reverse direction of the inequality above, therefore
    \begin{equation*}
        \left\|P_{2t,n}\right\|_{L^1\rightarrow L^\infty}= \left\|P_{t,n}\right\|_{L^2\rightarrow L^\infty}^2,
    \end{equation*}
    which is \eqref{eq:R_infty and R_2, reversible}. 
    
    By \parencite[Proposition 2.3]{chen2008cutoff}, $R_\infty$-cutoff is equivalent to $L^\infty$-cutoff, which is further equivalent to $L^2$ and $R_2$-cutoff by Proposition \ref{prop:characterization of L^p cutoff}, then we get the result. For discrete-time setting, the argument is similar.
\end{proof}

\subsection{$\alpha$-divergence and R\'enyi divergence with $1<\alpha\leq 2$}
\label{sec:alpha-divergence, 1<alpha<=2}

While the previous subsection \ref{sec:F_p,q family} investigates $\alpha$-divergence and R\'enyi divergence with $2<\alpha < \infty$, in this subsection we shall study equivalent conditions for cutoff phenomenon under $\alpha$-divergence and R\'enyi divergence with $1<\alpha\leq 2$. The technique in this part utilizes an argument about non-linear log-Sobolev/Poincar\'e inequalities (LSI/PI), which can be found in \parencite{chafai2004entropies,polyanskiy2019improved,varopoulos1985hardy,mossel2013reverse}.

\begin{definition}
    [Non-linear functional constants]
    \label{def:non-linear functional inequalities}
    Given an infinitesimal generator $\mathcal{A}$ and its associated Dirichlet form $\mathcal{E}_{\mathcal{A}}$ with $\pi$ as the stationary distribution, we define 
    \begin{align}
        (\text{Non-linear LSI})\qquad \rho(p)&:=\frac{p^2}{4(p-1)}\inf_{\mathrm{Ent}_\pi[f^p]>0}\frac{\mathcal{E}_{\mathcal{A}}\left(f,f^{p-1}\right)}{\mathrm{Ent}_\pi[f^p]},\label{eq:non-linear LSI}\\
        (\text{Non-linear PI})\qquad \lambda(p)&:=\frac{p^2}{4(p-1)}\inf_{\mathrm{Var}_\pi[f^{\frac{p}{2}}]>0}\frac{\mathcal{E}_{\mathcal{A}}\left(f,f^{p-1}\right)}{\mathrm{Var}_\pi[f^{\frac{p}{2}}]},\label{eq:non-linear PI}
    \end{align}
    where $p>0$, $p\neq 1$, $f>0$, $\mathrm{Ent}_\pi[f]:=\mathbb E_\pi\left[f\ln \frac{f}{\mathbb E_\pi[f]}\right]$ and $\mathrm{Var}_{\pi}[f]:=\mathbb E_\pi\left[\left(f-\mathbb E_\pi[f]\right)^2\right]$. We can also extend to case of $p=1$ by taking limits, i.e.
    \begin{equation*}
        \rho(1):=\inf_{\mathrm{Ent}_\pi[f]>0}\frac{\mathcal{E}_{\mathcal{A}}\left(f,\ln f\right)}{4\mathrm{Ent}_\pi[f]},\quad
        \lambda(1):=\inf_{\mathrm{Var}_\pi[f^{\frac{1}{2}}]>0}\frac{\mathcal{E}_{\mathcal{A}}\left(f,\ln f\right)}{4\mathrm{Var}_\pi[f^{\frac{1}{2}}]}.
    \end{equation*}
    In particular, it can be seen that $\rho(2)$ is the classical log-Sobolev constant, $\lambda(2)$ is the spectral gap or classical Poincar\'e constant and $\rho_n(1)$ is the classical modified log-Sobolev constant as in \parencite{bobkov2006modified}.
\end{definition}

\begin{lemma}\label{lemma: monotonicity of rho and lambda}
    The mappings $p\mapsto \rho(p)$ and $p\mapsto \lambda(p)$ are both non-increasing within $p\in (0,2]$, and non-decreasing within $p\in [2,\infty)$. 
\end{lemma}
\begin{proof}
    According to \parencite[Theorem 1.7]{mossel2013reverse}, $p\mapsto \rho(p)$ is non-increasing within $p\in (0,2]$. For $2\leq p<q<\infty$, it suffices to prove
    \begin{equation}\label{eq:proving monotonicity of rho(p) equivalent form}
        \frac{p^2}{4(p-1)}\inf_{\mathrm{Ent}_\pi[f^p]>0}\frac{\mathcal{E}_{\mathcal{A}}\left(f,f^{p-1}\right)}{\mathrm{Ent}_\pi[f^p]}\leq \frac{q^2}{4(q-1)}\inf_{\mathrm{Ent}_\pi[f^q]>0}\frac{\mathcal{E}_{\mathcal{A}}\left(f^{\frac{p}{q}},f^{\frac{q-1}{q}p}\right)}{\mathrm{Ent}_\pi[f^p]}
    \end{equation}
    by substituting $f^{\frac{q}{p}}$ to $f$ in the right hand side. For $g=f^p$, according to \parencite[Theorem 2.1]{mossel2013reverse}, we have 
    \begin{equation*}
        \frac{q^2}{q-1}\mathcal{E}_{\mathcal{A}}\left(g^{\frac{1}{q}}, g^{\frac{q-1}{q}}\right)\geq \frac{p^2}{p-1}\mathcal{E}_{\mathcal{A}}\left(g^{\frac{1}{p}}, g^{\frac{p-1}{p}}\right),
    \end{equation*}
    plugging into \eqref{eq:proving monotonicity of rho(p) equivalent form} we have $\rho(p)\leq \rho(q)$. The proof for the case of $\lambda(p)$ is similar.
\end{proof}

The non-linear functional constants introduced above will be used in the proof of our next result, which is presented in the following theorem.
\begin{theorem}
    [Characterization of $\alpha$-divergence and R\'enyi divergence cutoff for $1<\alpha\leq 2$]
    \label{thm:alpha-divergence cutoff with 1<alpha<2}
    Consider a sequence of Markov processes $\{X_t^{(n)},t\in T\}_{n=1}^\infty$ with state space $\mathcal{X}_n$, stationary distribution $\pi_n$, generator $\mathcal{A}_n$, spectral gap $\lambda_n\geq 0$, second largest singular value $0<\kappa_n\leq 1$ and semigroup $P_{t,n}$, where $P_{t,n}$ is \textbf{reversible} on $L^2(\mathcal{X}_n,\pi_n)$ for each $n\geq 1$.  If $T=[0,\infty)$, let $g_n(t):=\widetilde d_{f_{\alpha},n}(t)$ and assume $\lim_{t\rightarrow \infty}g_n(t)=0$ for each $n$, then the following statements are equivalent:
    \begin{enumerate}[label=(D\arabic*)]
        \item\label{item:D1} There exists some $1<\alpha\leq 2$ and some $\varepsilon>0$ such that $\lambda_n\widetilde t_{f_\alpha,n}(\varepsilon)\rightarrow \infty$. 

        \item\label{item:D2}For any $1<\alpha\leq 2$ and any $\varepsilon>0$, $\lambda_n\widetilde t_{f_\alpha,n}(\varepsilon)\rightarrow \infty$.

        \item\label{item:D3} For any $1<\alpha\leq 2$ and any $\varepsilon>0$, precutoff occurs.

        \item\label{item:D4} For any $1<\alpha\leq 2$ and any $\varepsilon>0$, cutoff occurs.

        \item\label{item:D5} For any $1<\alpha\leq 2$ and any $\varepsilon>0$, there is a $\left(\widetilde t_{f_\alpha,n}(\varepsilon),\lambda_n^{-1}\right)$ cutoff. 
    \end{enumerate}
    Meanwhile, if we take $g_n(t):=\widetilde d_{R_{\alpha},n}(t)$ and assume $\lim_{t\rightarrow \infty}g_n(t)=0$ for each $n$, then the following statements are equivalent:
    \begin{enumerate}[label=(D\arabic*')]
        \item\label{item:D1'} There exists some $1<\alpha\leq 2$ and some $\varepsilon>0$ such that $\lambda_n\widetilde t_{R_\alpha,n}(\varepsilon)\rightarrow \infty$. 

        \item\label{item:D2'} For any $1<\alpha\leq 2$ and any $\varepsilon>0$, $\lambda_n\widetilde t_{R_\alpha,n}(\varepsilon)\rightarrow \infty$.

        \item\label{item:D3'} For any $1<\alpha\leq 2$ and any $\varepsilon>0$, precutoff occurs.

        \item\label{item:D4'} For any $1<\alpha\leq 2$ and any $\varepsilon>0$, cutoff occurs.

        \item\label{item:D5'} For any $1<\alpha\leq 2$ and any $\varepsilon>0$, there is a $\left(\widetilde t_{R_\alpha,n}(\varepsilon),\lambda_n^{-1}\right)$ cutoff. 
    \end{enumerate}
    Moreover, items \ref{item:D1} to \ref{item:D5} are equivalent to items \ref{item:D1'} to \ref{item:D5'}. 
    
    For $T=\mathbb N$, 
    assume for some $1<\alpha\leq 2$ and some $\varepsilon>0$, $\lim_{n\rightarrow \infty}\widetilde t_{f_\alpha,n}(\varepsilon)=\infty$ and $\lim_{n\rightarrow \infty}\widetilde t_{R_\alpha,n}(\varepsilon)=\infty$ respectively. If we substitute $\lambda_n'=\min \{1, -\ln \kappa_n\}$ into $\lambda_n$ in the items and assume $\kappa_n\rightarrow 1$, then the statements above also hold. Besides, if $\lambda_n\rightarrow 0$ and the chains are lazy, we can also take $\lambda_n'=\min \left\{1, \lambda_n\right\}$.
\end{theorem}

\begin{remark}
    Theorem \ref{thm:alpha-divergence cutoff with 1<alpha<2} and \ref{thm:Renyi cutoff alpha >= 2} have a common ground of the $R_2$ divergence for $\alpha=2$, which indicates that items \ref{item:D1} to \ref{item:D5} and items \ref{item:D1'} to \ref{item:D5'} are all equivalent to items \ref{item:A1} to \ref{item:A5}, \ref{item:B1} to \ref{item:B5} and \ref{item:C1} to \ref{item:C5}, and hence we can still call $\alpha$-divergence and R\'enyi divergence with $1<\alpha<2$ as \boldmath$L^2$\textbf{-type} divergences under cutoff phenomenon, although some mild assumptions like \eqref{eq:1/2 lazy chains} may be added in discrete-time case.
\end{remark}

\begin{proof}
    We first consider $T=[0,\infty)$. We first show that \ref{item:D2}$\Rightarrow$\ref{item:D5}$\Rightarrow$\ref{item:D4}$\Rightarrow$\ref{item:D3}$\Rightarrow$\ref{item:D2}, then the case for \ref{item:D2'}$\Rightarrow$\ref{item:D5'}$\Rightarrow$\ref{item:D4'}$\Rightarrow$\ref{item:D3'}$\Rightarrow$\ref{item:D2'} is similar using the idea in the proof of Theorem \ref{thm:Renyi cutoff alpha >= 2}. Finally we prove \ref{item:D1'}$\iff$\ref{item:D2'} to \ref{item:D5'}, and \ref{item:D1}$\iff$\ref{item:D1'} to \ref{item:D5'}$\iff$ \ref{item:D1} to \ref{item:D5} to complete the proof.

    \ref{item:D2}$\Rightarrow$\ref{item:D5}: Given $1<\alpha\leq 2$ and $\varepsilon>0$, we use $h_{t,n}^x(y)=h_n(t,x,y)$ to represent the probability density function of $\delta_xP_{t,n}$ with respect to $\pi_n$ as shown in \eqref{eq:L^p divergence}, and hence 
    \begin{equation*}
        d_{f_\alpha,n}(x,t)=\frac{1}{\alpha-1}\left(\int_{\mathcal{X}_n}h_n^\alpha(t,x,y)\pi_n(dy)-1\right)=\frac{\left\|h_n(t,x,\cdot)\right\|_\alpha^{\alpha}-1}{\alpha-1}.
    \end{equation*}
    Taking differentiation with respect to $t$ and use \eqref{eq:time volution of h_t}, we have 
    \begin{align*}
        \frac{\partial}{\partial t} d_{f_\alpha,n}(x,t)&=\frac{\alpha}{\alpha-1}\int_{\mathcal{X}_n}h_n^{\alpha-1}(t,x,y)\frac{\partial}{\partial t}h_n(t,x,y)\pi_n(dy)\\
        &=\frac{\alpha}{\alpha-1}\int_{\mathcal{X}_n}\left(h_{t,n}^{x}(y)\right)^{\alpha-1}\mathcal{A}_n^*h_{t,n}^x(y)\pi_n(dy)\\
        &=-\frac{\alpha}{\alpha-1}\mathcal{E}_{\mathcal{A}_n}\left(h_{t,n}^x,\left(h_{t,n}^x\right)^{\alpha-1}\right)\\
        &=-\frac{4(\alpha-1)}{\alpha}\cdot\frac{\alpha^2}{4(\alpha-1)}\cdot\frac{\mathcal{E}_{\mathcal{A}_n}\left(h_{t,n}^x,\left(h_{t,n}^x\right)^{\alpha-1}\right)}{\left\|h_{t,n}^x\right\|_\alpha^\alpha-1}\cdot d_{f_\alpha,n}(x,t).
    \end{align*}
    Since $\mathbb E_\pi\left[\left(h_{t,n}^x\right)^{\alpha/2}\right]\leq \mathbb E_\pi\left[h_{t,n}^x\right]^{\alpha/2}=1$ for $1<\alpha\leq 2$ by H\"older's inequality, recalling the non-linear Poincar\'e constant defined in \eqref{eq:non-linear PI}, we have 
    \begin{align*}
        \frac{\partial}{\partial t} d_{f_\alpha,n}(x,t)&\leq -\frac{4(\alpha-1)}{\alpha}\cdot\frac{\alpha^2}{4(\alpha-1)}\cdot\frac{\mathcal{E}_{\mathcal{A}_n}\left(h_{t,n}^x,\left(h_{t,n}^x\right)^{\alpha-1}\right)}{\mathrm{Var}_{\pi_n}\left[\left(h_{t,n}^x\right)^{\frac{\alpha}{2}}\right]}\cdot d_{f_\alpha,n}(x,t)\\
        &\leq -\frac{4(\alpha-1)}{\alpha}\cdot \lambda_n(\alpha)\cdot d_{f_\alpha,n}(x,t),
    \end{align*}
    where $\lambda_n(\alpha)$ denotes the non-linear Poincar\'e constant of the $n^{th}$ process. According to Lemma \ref{lemma: monotonicity of rho and lambda}, we have 
    \begin{equation*}
        \lambda_n(\alpha)\geq \lambda_n(2)=\lambda_n, \quad 1<\alpha\leq 2,
    \end{equation*}
    therefore
    \begin{equation*}
        \frac{\partial}{\partial t} d_{f_\alpha,n}(x,t)\leq -\frac{4(\alpha-1)}{\alpha}\cdot \lambda_n \cdot d_{f_\alpha,n}(x,t),
    \end{equation*}
    integrating from $t=u$ to $t=u+v$ for any $u,v\geq 0$ yields
    \begin{equation}\label{eq:convergene rate of alpha divergence when 1<alpha<2}
        d_{f_\alpha,n}(x,u+v)\leq d_{f_\alpha,n}(x,u)\exp \left(-\frac{4(\alpha-1)}{\alpha}\cdot \lambda_n v\right).
    \end{equation}
    Taking supremum over $x\in \mathcal{X}_n$ and use the same argument in the proof of \ref{item:B2}$\Rightarrow$\ref{item:B5} we get the result.

    \ref{item:D5}$\Rightarrow$\ref{item:D4}$\Rightarrow$\ref{item:D3}: By definition.

    \ref{item:D3}$\Rightarrow$\ref{item:D2}: Recalling Proposition \ref{prop:properties of f-divergence} items \ref{item:monotonicity of Renyi divergence} and \ref{item:Pinsker's inequality}, for any given $1<\alpha\leq 2$, we have 
    \begin{equation*}
        d_{f_\alpha,n}(x,t)\geq \mathrm{KL}\left(\delta_xP_{t,n}\|\pi_n\right)\geq 2\mathrm{TV}^2(\delta_xP_{t,n},\pi_n)=\frac{1}{2}d_{1,n}^2(x,t).
    \end{equation*}
    Then by Proposition \ref{prop:Important properties of L^p}, we have 
    \begin{equation*}
        \widetilde d_{f_\alpha,n}(t)\geq \frac{1}{2}\widetilde d_{1,n}^2(t)=\frac{1}{2}\left\|P_{t,n}-\Pi_n\right\|_{L^\infty\rightarrow B}^2\geq \frac{1}{2}\left\|P_{t,n}-\Pi_n\right\|_{L^\infty\rightarrow L^\infty}^2.
    \end{equation*}
    Recalling that $\left\|P_{t,n}-\Pi_n\right\|_{L^2\rightarrow L^2}=e^{-\lambda_nt}$ as stated in Proposition \ref{prop: convergence rate of Markov semigroup}, using Riesz-Thorin Interpolation Theorem we have 
    \begin{align*}
        e^{-\lambda_nt}=\left\|P_{t,n}-\Pi_n\right\|_{L^2\rightarrow L^2}&\leq \left\|P_{t,n}-\Pi_n\right\|_{L^1\rightarrow L^1}^{\frac{1}{2}}\left\|P_{t,n}-\Pi_n\right\|_{L^\infty\rightarrow L^\infty}^{\frac{1}{2}}\\
        &\leq 2^{\frac{1}{2}}\left\|P_{t,n}-\Pi_n\right\|_{L^\infty\rightarrow L^\infty}^{\frac{1}{2}},
    \end{align*}
    hence
    \begin{equation*}
        \widetilde t_{f_\alpha,n}(t)\geq \frac{1}{8}e^{-4\lambda_nt},
    \end{equation*}
    then by similar argument in the proof of \ref{item:B3}$\Rightarrow$\ref{item:B2} we get the result. 

    Equivalence within items \ref{item:D2'} to \ref{item:D5'}: Similar to the proof in Theorem \ref{thm:Renyi cutoff alpha >= 2}. 

    \ref{item:D1'}$\Rightarrow$\ref{item:D2'}: We follow the proof of \parencite[Proposition 5.1]{chen2008cutoff}. Suppose for some $1<\alpha\leq 2$ and $\varepsilon>0$, $\lambda_n \widetilde t_{R_\alpha,n}(\varepsilon)\rightarrow\infty$. We first recall  that R\'enyi divergence can be written as the following form:
    \begin{equation}\label{eq:norm form of Renyi divergence}
        d_{R_\alpha,n}(x,t)=\frac{\alpha}{\alpha-1}\ln \left\|h_{t,n}^x\right\|_\alpha,
    \end{equation}
    and that by Proposition \ref{prop:Important properties of L^p},
    \begin{equation*}
        \left\|h_{t,n}^x\right\|_\alpha=\sup \left\{\mu_{t,n}^x(g):\|g\|_{\alpha'}\leq 1\right\},
    \end{equation*}
    where we denote $\mu_{t,n}^x:=\delta_xP_{t,n}$, and $\alpha'$ as the conjugate of $\alpha$, i.e. $\frac{1}{\alpha}+\frac{1}{\alpha'}=1$. Then, for $t=u+v$ with $u,v\geq 0$ and any $g\in L^{\alpha'}(\mathcal{X}_n,\pi_n)$, we have 
    \begin{equation*}
        \mu_{t,n}^x(g)=\mu_{u,n}^xP_{v,n}(g),
    \end{equation*}
    hence by H\"older's inequality, for any given $1<\beta<\alpha$,
    \begin{equation*}
        \left|\mu_{t,n}^x(g)\right|\leq \left\|h_{u,n}^x\right\|_{\beta}\left\|P_{v,n}(g)\right\|_{\beta'}\leq \left\|h_{u,n}^x\right\|_{\beta}\left\|P_{v,n}\right\|_{L^{\alpha'}\rightarrow L^{\beta'}},
    \end{equation*}
    where $\frac{1}{\beta}+\frac{1}{\beta'}=1$. Taking supremum over $g\in L^{\alpha'}(\mathcal{X}_n,\pi_n)$ yields
    \begin{equation}\label{eq:Renyi divergence Holder decomposition}
        \left\|h_{t,n}^x\right\|_\alpha\leq \left\|h_{u,n}^x\right\|_{\beta}\left\|P_{v,n}\right\|_{L^{\alpha'}\rightarrow L^{\beta'}}.
    \end{equation}
    For some $1<\gamma<\alpha$ with $\frac{1}{\gamma}+\frac{1}{\gamma'}=1$ which satisfies $\frac{1}{\alpha'}=\frac{1}{\beta'}+\frac{1}{\gamma'}$ or equivalently $1+\frac{1}{\alpha}=\frac{1}{\beta}+\frac{1}{\gamma}$, by the Riesz-Thorin Interpolation Theorem we have
    \begin{align*}
        \left\|P_{v,n}\right\|_{L^{\alpha'}\rightarrow L^{\beta'}}&\leq \left\|P_{v,n}\right\|_{L^1\rightarrow L^{\gamma}}^{\gamma/\beta'}\left\|P_{v,n}\right\|_{L^{\gamma'}\rightarrow L^{\infty}}^{1-\gamma/\beta'}\\
        &=\left\|P_{v,n}^*\right\|_{L^{\gamma'}\rightarrow L^{\infty}}^{\gamma/\beta'}\left\|P_{v,n}\right\|_{L^{\gamma'}\rightarrow L^{\infty}}^{1-\gamma/\beta'}\\
        &=\left\|P_{v,n}\right\|_{L^{\gamma'}\rightarrow L^{\infty}},
    \end{align*}
    where we recall that the processes are reversible. Therefore, we have
    \begin{equation*}
        \left\|h_{t,n}^x\right\|_\alpha\leq \left\|h_{u,n}^x\right\|_{\beta}\left\|P_{v,n}\right\|_{L^{\gamma'}\rightarrow \infty}. 
    \end{equation*}
    Plugging the above into \eqref{eq:norm form of Renyi divergence}, and observe that by Proposition \ref{prop:Important properties of L^p},
    \begin{equation}\label{eq:expression of overline Renyi divergence}
        \widetilde d_{R_\alpha,n}(t)=\frac{\alpha}{\alpha-1}\ln \|P_{t,n}\|_{L^{\alpha'}\rightarrow \infty},
    \end{equation}
    we obtain
    \begin{equation*}
        \widetilde d_{R_\alpha,n}(u+v)\leq \frac{\alpha'}{\beta'}\widetilde d_{R_\beta,n}(u)+\frac{\alpha'}{\gamma'}\widetilde d_{R_\gamma,n}(v),
    \end{equation*}
    which implies for any $\varepsilon_1,\varepsilon_2>0$,
    \begin{equation*}
        \widetilde t_{R_\alpha,n}\left(\frac{\alpha'}{\beta'}\varepsilon_1+\frac{\alpha'}{\gamma'}\varepsilon_2\right)\leq \widetilde t_{R_\beta,n}(\varepsilon_1)+\widetilde t_{R_\gamma,n}(\varepsilon_2).
    \end{equation*}
   If we further take $\varepsilon_1=\varepsilon_2=\varepsilon$,
    \begin{equation}\label{eq:relationship between mixing time of Renyi divergence for alpha<2}
        \widetilde t_{R_\alpha,n}\left(\varepsilon\right)\leq \widetilde t_{R_\beta,n}(\varepsilon)+\widetilde t_{R_\gamma,n}(\varepsilon).
    \end{equation}
    Note that the argument above does not have specific requirements for $\alpha$ except for $\alpha>1$, we can take 
    \begin{equation*}
        \frac{1}{\alpha_i}=\left(1-\frac{1}{\beta}\right)i+\frac{1}{\alpha}, \quad i=0,1,2,...
    \end{equation*}
    which satisfies $1+\frac{1}{\alpha_i}=\frac{1}{\alpha_{i+1}}+\frac{1}{\beta}$, $\alpha_0=\alpha$ and $\alpha_i\leq \beta$ if and only if $i+1\geq \frac{\beta'}{\alpha'}$. Combining with \eqref{eq:relationship between mixing time of Renyi divergence for alpha<2}, we get 
    \begin{equation*}
        \widetilde t_{R_{\alpha_i},n}\left(\varepsilon\right)\leq \widetilde t_{R_\beta,n}(\varepsilon)+\widetilde t_{R_{\alpha_{i+1}},n}(\varepsilon), \quad i=0,1,..., \left\lceil \frac{\beta'}{\alpha'}\right\rceil-1,
    \end{equation*}
    taking summation we have
    \begin{equation}\label{eq:relationship between mixing time of Renyi divergence with beta<alpha<2}
        \widetilde t_{R_\beta,n}(\varepsilon)\leq\widetilde t_{R_\alpha,n}\left(\varepsilon\right)\leq \left(\left\lceil \frac{\beta'}{\alpha'}\right\rceil+1\right)\widetilde t_{R_\beta,n}(\varepsilon),
    \end{equation}
    where the left inequality comes from the monotonicity of R\'enyi divergence as shown in Proposition \ref{prop:properties of f-divergence} item \ref{item:monotonicity of Renyi divergence}. Recalling that by \parencite[Corollary 2.5]{chen2008cutoff}, \ref{item:D1} implies for any $\delta>0$, there is a $\left(\widetilde t_{R_\alpha,n}(\delta),\lambda_n^{-1}\right)$ cutoff, which further yields for any $\delta>0$, $\lambda_n\widetilde t_{R_\alpha,n}(\delta)\rightarrow \infty$, then by \eqref{eq:relationship between mixing time of Renyi divergence with beta<alpha<2} we obtain that for any $1<\beta<\alpha$ and any $\delta>0$, $\lambda_n\widetilde t_{R_\beta,n}(\delta)\rightarrow \infty$. For $\alpha\leq \beta\leq 2$, the result also holds via monotonicity of R\'enyi divergence.

    \ref{item:D1}$\iff$\ref{item:D1'} to \ref{item:D5'}$\iff$\ref{item:D2} to \ref{item:D5}: Use the identity \eqref{eq:identity between alpha-divergence and Renyi divergence}. 

    Next, we consider $T=\mathbb N$. The discrete-time case for $\alpha$-divergence is not that trivial compared with $L^p$-cutoff, since the techniques in the continuous-time case involve taking derivative with respect to $t$ which may not apply in discrete-time setting. We only need to give an adapted version of the proof in \ref{item:D2}$\Rightarrow$\ref{item:D5}, and the rest are similar to the continuous-time case. The proof is inspired by \parencite[Proposition 6]{miclo1997}. In this case, the Markov semigroup satisfies $P_{k,n}=P_{1,n}^k$, and we denote the one-step transition probability of adjoint generator as $p_n^*(x,\cdot)$. 

    We observe that for any $t,s\geq 0$ and any $1<\alpha\leq 2$, 
    \begin{equation}\label{eq:alpha-divergence Taylor type inequality with 1<alpha<=2}
        \frac{t^\alpha-1}{\alpha-1}\geq \frac{s^\alpha-1}{\alpha-1}+\frac{\alpha s^{\alpha-1}}{\alpha-1}(t-s)+\left(t^{\frac{\alpha}{2}}-s^{\frac{\alpha}{2}}\right)^2,
    \end{equation}
    which can be easily verified via taking derivative with respect to $\frac{t}{s}\geq 0$. For any $\mu_{0,n}$ on $\mathcal{X}_n$, let $f=\frac{d\mu_{0,n}}{d\pi_n}$ be the initial density for simplicity of notation, and we have $P_{1,n}^*f=\frac{d\mu_{0,n}P_{1,n}}{d\pi_n}$. Plug $t=f(z)$ and $s=P_{1,n}^*f(y)$ into \eqref{eq:alpha-divergence Taylor type inequality with 1<alpha<=2}, and take expectation on both sides with respect to $p_n^*(y,dz)$, we have 
    \begin{align*}
        P_{1,n}^*\left(\frac{f^\alpha-1}{\alpha-1}\right)(y)&\geq \frac{\left(P_{1,n}^*f(y)\right)^\alpha-1}{\alpha-1}+\int_{z\in \mathcal{X}_n}\left(f(z)^{\frac{\alpha}{2}}-\left(P_{1,n}^*f(y)\right)^{\frac{\alpha}{2}}\right)^2p_n^*(y,dz)\\
        &\geq \frac{\left(P_{1,n}^*f(y)\right)^\alpha-1}{\alpha-1}+\mathrm{Var}_{p^*(y,\cdot)}\left[f^{\frac{\alpha}{2}}\right]\\
        &=\frac{\left(P_{1,n}^*f(y)\right)^\alpha-1}{\alpha-1}+P_{1,n}^*\left(f^\alpha\right)(y)-\left(P_{1,n}^*\left(f^{\frac{\alpha}{2}}\right)(y)\right)^2.
    \end{align*}
    Taking expectation on $y$ with respect to $\pi_n$, we have 
    \begin{equation}\label{eq:convergence of d_alpha in discrete time}
        d_{f_\alpha,n}\left(\mu_{0,n}\|\pi_n\right)\geq d_{f_\alpha,n}\left(\mu_{0,n}P_{1,n}\|\pi_n\right)+\left\langle f^{\frac{\alpha}{2}},\left(I-P_{1,n}P_{1,n}^*\right)f^{\frac{\alpha}{2}}\right\rangle_{\pi_n}.
    \end{equation}
    Similar to the proof in the continuous case, we have 
    \begin{equation}\label{eq:comparing Dirichlet form and d_alpha in discrete time}
        \frac{\left\langle f^{\frac{\alpha}{2}},\left(I-P_{1,n}P_{1,n}^*\right)f^{\frac{\alpha}{2}}\right\rangle_{\pi_n}}{d_{f_\alpha,n}\left(\mu_{0,n}\|\pi_n\right)}\geq (\alpha-1)\frac{\left\langle f^{\frac{\alpha}{2}},\left(I-P_{1,n}P_{1,n}^*\right)f^{\frac{\alpha}{2}}\right\rangle_{\pi_n}}{\mathrm{Var}_{\pi_n}\left[f^\frac{\alpha}{2}\right]}
        \geq (\alpha-1)\lambda\left(P_{1,n}P_{1,n}^*-I\right),
    \end{equation}
    where $\lambda\left(P_{1,n}P_{1,n}^*-I\right)$ is the spectral gap of the generator $P_{1,n}P_{1,n}^*-I$. If the Markov chains are lazy and $\lambda_n\rightarrow 0$, denote $q_n(x,\cdot)$ as the transition probability of $P_{1,n}P_{1,n}^*$. Hence, we have 
    \begin{align*}
        \frac{\pi_n(dx)q_n(x,dy)}{dxdy}&\geq \frac{\pi_n(dx)p_n(x,\{x\})p_n^*(x,dy)}{dxdy}\geq \frac{1}{2}\frac{\pi_n(dx)p_n^*(x,dy)}{dxdy}\\
        &=\frac{1}{2}\frac{\pi_n(dy)p_n(y,dx)}{dxdy},
    \end{align*}
    where the second inequality comes from the laziness \eqref{eq:1/2 lazy chains}. Recalling \eqref{eq:expression of Dirichlet f,f}, we have 
    \begin{equation}\label{eq:comparing Dirichlet form of PP* and P}
        \left\langle f^{\frac{\alpha}{2}},\left(I-P_{1,n}P_{1,n}^*\right)f^{\frac{\alpha}{2}}\right\rangle_{\pi_n}\geq \frac{1}{2}\left\langle f^{\frac{\alpha}{2}},\left(I-P_{1,n}\right)f^{\frac{\alpha}{2}}\right\rangle_{\pi_n},
    \end{equation}
    therefore 
    \begin{equation*}
        \lambda\left(P_{1,n}P_{1,n}^*-I\right)\geq \frac{1}{2}\lambda\left(P_{1,n}-I\right)=\frac{1}{2}\lambda_n,
    \end{equation*}
    combined with \eqref{eq:convergence of d_alpha in discrete time} and \eqref{eq:comparing Dirichlet form and d_alpha in discrete time} we obtain 
    \begin{align}
        d_{f_\alpha,n}\left(\mu_{0,n}P_{1,n}\|\pi_n\right)&\leq \left(1-\frac{\alpha-1}{2}\lambda_n\right)d_{f_\alpha,n} \left(\mu_{0,n}\|\pi_n\right)\nonumber\\
        &\leq \exp \left(-\frac{\alpha-1}{2}\lambda_n\right)d_{f_\alpha,n} \left(\mu_{0,n}\|\pi_n\right).\label{eq:convergence ratio of alpha-divergence, lazy condition}
    \end{align}
    On the other hand, if we use the condition of $\kappa_n$, recalling that $P_{1,n}=P_{1,n}^*$, we have 
    \begin{equation*}
        \lambda\left(P_{1,n}P_{1,n}^*-I\right)=1-\kappa_n^2,
    \end{equation*}
    hence
    \begin{align}
        d_{f_\alpha,n}\left(\mu_{0,n}P_{1,n}\|\pi_n\right)&\leq \left(1-(\alpha-1)(1-\kappa_n^2)\right)d_{f_\alpha,n} \left(\mu_{0,n}\|\pi_n\right)\nonumber\\
        &\leq \exp\left(-(\alpha-1)(1-\kappa_n^2)\right)d_{f_\alpha,n} \left(\mu_{0,n}\|\pi_n\right)\nonumber\\
        &\leq \exp\left(2(\alpha-1)\kappa_n^2\ln \kappa_n\right)d_{f_\alpha,n} \left(\mu_{0,n}\|\pi_n\right),\label{eq:convergence ratio of alpha-divergence, kappa condition}
    \end{align}
    where we have used $1-x\geq -x\ln x$ for $x\geq 0$ in the last inequality. For \eqref{eq:convergence ratio of alpha-divergence, lazy condition}, the desired result follows directly by using similar argument as in the continuous-time setting. For \eqref{eq:convergence ratio of alpha-divergence, kappa condition}, since $\kappa_n\rightarrow 1$, we get the result.
\end{proof}

\begin{remark}
    \eqref{eq:relationship between mixing time of Renyi divergence with beta<alpha<2} indicates that mixing times of R\'enyi divergences, and hence $\alpha$-divergences, are equivalent for different $\alpha\in (1,\infty)$. 
\end{remark}

With the discussions before, we have a direct corollary of a common sufficient condition of cutoff phenomenon under all the $L^2$-type divergences in continuous-time finite state space. In discrete-time case, results similar to the following may not hold, and readers can check \parencite[Remark 4.14]{montenegro2006mathematical} for counterexamples.

\begin{corollary}
    [Common sufficient condition in terms of spectral gap and log-Sobolev constant, continuous-time, finite state space]
    \label{corollary:spectral times inverse mLSI, continuous time}
    According to \parencite[Corollary 3.11]{diaconis1996logarithmic}, under the assumption of \textbf{continuous-time finite state space}, suppose the $n^{th}$ Markov process has log-Sobolev constant $\rho_n(2)$, then 
    \begin{equation*}
        \widetilde t_{2,n}\left(1/e\right)\geq \frac{1}{2\rho_n(2)},
    \end{equation*}
    which implies that the common sufficient condition for cutoff phenomenon for all the $L^2$-type divergences under this situation is 
    \begin{equation}\label{eq:common sufficient condition in terms of spectral gap and LSI}
        \lambda_n\rho_n(2)^{-1}\rightarrow \infty, \quad \text{as }n\rightarrow \infty.
    \end{equation}
\end{corollary}

\begin{remark}
    The advantage of Corollary \ref{corollary:spectral times inverse mLSI, continuous time} lies in relaxing the requirement for the lower bound of mixing times, although extra knowledge of log-Sobolev constant is needed to check \eqref{eq:common sufficient condition in terms of spectral gap and LSI}. There have been several numerical methods to approximately determine the log-Sobolev constant of a given Markov chain, for example the semidefinite programming with sum-of-squares method \parencite{faust2023sum}.
\end{remark}

\subsection{KL divergence and total variation distance}\label{sec:pi-weighted KL and TV}
In previous subsections, we have pointed out the equivalence of $\alpha$-divergence or R\'enyi divergence within $\alpha\in (1,\infty)$ under cutoff phenomenon. However, in the special case when $\alpha=1$ such that the divergence reduces to the KL divergence, the criterion for cutoff is more delicate. Generally the sufficient condition and necessary condition for cutoff may not be the same, and we will give an explicit explanation in the sequel, where the worst-case KL divergence and $\pi$-weighted KL divergence or TV distance are considered.

Analogous to the notations introduced earlier, we write that 
\begin{equation*}
    d_{\mathrm{KL}}(x,t):=d_{f_1}(x,t)=\mathrm{KL}\left(\delta_x P_t\|\pi\right), \quad
    \widetilde d_{\mathrm{KL}}(t):=\pi\text{-}\mathop{\esssup}\limits_{x\in\mathcal{X}}d_{\mathrm{KL}}(x,t),
\end{equation*}
and the mixing time
\begin{equation*}
\widetilde t_{\mathrm{KL}}(\varepsilon):=\inf\left\{t\in T:\widetilde d_{\mathrm{KL}}(t)\leq \varepsilon\right\},    
\end{equation*}
based on which we have the following result derived from non-linear log-Sobolev inequality as introduced in Definition \ref{def:non-linear functional inequalities}.

\begin{theorem}
    [Characterization of KL divergence cutoff via modified LSI]
    \label{thm:KL divergence cutoff via LSI}
    Recall the functional constants defined in Definition \ref{def:non-linear functional inequalities}. Consider a sequence of Markov processes $\{X_t^{(n)},t\in T\}_{n=1}^\infty$ with state space $\mathcal{X}_n$, stationary distribution $\pi_n$, generator $\mathcal{A}_n$, spectral gap $\lambda_n\geq 0$, modified log-Sobolev constant $\rho_n(1)\geq 0$ and semigroup $P_{t,n}$, where $P_{t,n}$ is \textbf{reversible} on $L^2(\mathcal{X}_n,\pi_n)$ for each $n\geq 1$.  If $T=[0,\infty)$, let $g_n(t):=\widetilde d_{\mathrm{KL},n}(t)$ and assume $\lim_{t\rightarrow \infty}g_n(t)=0$ for each $n$, then the following statements hold:
    \begin{enumerate}[label=(\roman*)]
        \item\label{item:KL via LSI sufficient} For any $\varepsilon>0$, if $\rho_n(1)\cdot\widetilde t_{\mathrm{KL},n}(\varepsilon)\rightarrow \infty$, then there is a $\left(\widetilde t_{\mathrm{KL},n}(\varepsilon),\rho_n^{-1}(1)\right)$ cutoff. 

        \item\label{item:KL via LSI necessary} For any $\varepsilon>0$, if precutoff occurs, then $\lambda_n\widetilde t_{\mathrm{KL},n}(\varepsilon)\rightarrow \infty$.
    \end{enumerate}
    For $T=\mathbb N$, assume the Markov chains are lazy. Let $\lambda_n\rightarrow 0$, and for some $\varepsilon>0$, $\lim_{n\rightarrow \infty}\widetilde t_{\mathrm{KL},n}(\varepsilon)=\infty$. If we substitute $\lambda_n'=\min \{1, \lambda_n\}$ into $\lambda_n$ in the items, then the statements above also hold. 
\end{theorem}
\begin{proof}
    We first consider the continuous-time case.
    
    \ref{item:KL via LSI sufficient}: For $h_{t,n}^x=\frac{d\delta_xP_t}{d\pi}$, recalling that 
    \begin{equation*}
        d_{\mathrm{KL},n}(x,t)=\int_{\mathcal{X}_n}h_{t,n}^x(y)\ln h_{t,n}^x(y) \pi_n(dy)=\mathrm{Ent}_{\pi_n}\left[h_{t,n}^x\right],
    \end{equation*}
    if we differentiate with respect to $t$ and follow by using \eqref{eq:time volution of h_t}, then we have
    \begin{align*}
        \frac{\partial}{\partial t}d_{\mathrm{KL},n}(x,t)&=\int_{\mathcal{X}_n}\frac{\partial}{\partial t} \left(h_{t,n}^x(y)\ln h_{t,n}^x(y)\right) \pi_n(dy)\\
        &=\int_{\mathcal{X}_n} \left(\mathcal{A}_n^*h_{t,n}^x(y)\ln h_{t,n}^x(y)+\mathcal{A}_n^*h_{t,n}^x(y)\right)\pi_n(dy)\\
        &=-\mathcal{E}_{\mathcal{A}_n}\left(h_{t,n}^x,\ln h_{t,n}^x\right)\\
        &=-4\cdot\frac{\mathcal{E}_{\mathcal{A}_n}\left(h_{t,n}^x,\ln h_{t,n}^x\right)}{4\mathrm{Ent}_{\pi_n}\left[h_{t,n}^x\right]}\cdot d_{\mathrm{KL},n}(x,t)\\
        &\leq -4\rho_n(1)\cdot d_{\mathrm{KL},n}(x,t),
    \end{align*}
    where $\rho_n(1)$ is defined in Definition \ref{def:non-linear functional inequalities}. Taking integration from $t=u$ to $t=u+v$ with $u,v\geq 0$ yields
    \begin{equation*}
        d_{\mathrm{KL},n}(x,u+v)\leq e^{-4\rho_n(1)\cdot v}d_{\mathrm{KL},n}(x,u),
    \end{equation*}
    taking supremum over $x\in \mathcal{X}_n$ and use the same argument in the proof of \ref{item:B2}$\Rightarrow$\ref{item:B5}, we get the result.

    \ref{item:KL via LSI necessary}: Similar to the proof of \ref{item:D3}$\Rightarrow$\ref{item:D2}.

    If $T=\mathbb N$, similar to \eqref{eq:comparing Dirichlet form of PP* and P}, since $f$ and $\ln f$ have same monotonicity, we have 
    \begin{equation*}
        \rho\left(I-P_{1,n}P_{1,n}^*,1\right)\geq \frac{1}{2}\rho\left(I-P_{1,n},1\right)=\frac{1}{2}\rho_n(1),
    \end{equation*}
    where $\rho\left(I-P_{1,n}P_{1,n}^*,1\right)$ is the modified log-Sobolev constant of the generator $I-P_{1,n}P_{1,n}^*$. Then the result is direct to obtain via \parencite[Proposition 6]{miclo1997}.
\end{proof}

\begin{remark}
    There is a slight difference in the discrete-time case compared with results in Section \ref{sec:F_p,q family} and \ref{sec:alpha-divergence, 1<alpha<=2}: We do not have result in terms of second largest singular value, since in the proof of sufficient condition, we have used modified log-Sobolev constant instead of spectral gap.
\end{remark}

In Theorem \ref{thm:KL divergence cutoff via LSI}, we have used log-Sobolev constant to characterize the cutoff phenomenon for KL divergence, which is inconsistent with the necessary condition yet. In finite Markov chains, the log-Sobolev constant $\rho$ usually has a lower bound related to the parameters of whole state space $\mathcal{X}$, stationary distribution $\pi$ and transition matrix, for instance $\pi_{\mathrm{min}}:=\min_{x\in \mathcal{X}}\pi(x)$, $|\mathcal{X}|$, diameter of the state space and so on. For example, in \parencite[Corollary 4.15]{montenegro2006mathematical} it is stated that 
\begin{equation}\label{eq:classical result of lower bound of LSI in terms of PI}
    \rho\geq \frac{\lambda}{2+\ln \frac{1-\pi_{\mathrm{min}}}{\pi_{\mathrm{min}}}}\geq \frac{\lambda}{2+\ln \frac{1}{\pi_{\mathrm{min}}}}.
\end{equation}
We also refer readers to \parencite{cryan2021modified} where extra conditions on $\pi$ are proposed. In general state space Markov processes, similar lower bounds also exist using other quantities of the process, for example in \parencite{wang1997logarithmic}. These results can imply a lower bound of $\rho$ in terms of spectral gap $\lambda$, which further suggests possible criteria for KL-cutoff. As a concrete example, \eqref{eq:classical result of lower bound of LSI in terms of PI} implies that a sufficient condition for KL divergence cutoff is 
\begin{equation}\label{eq:classical KL sufficient condition in terms of spectral gap}
   \frac{\lambda_n\widetilde t_{\mathrm{KL},n}(\varepsilon)}{2+\ln \frac{1}{\pi_{\mathrm{min},n}}}\rightarrow \infty, \quad \text{as }n\rightarrow\infty.
\end{equation}
Moreover, curvature is also a useful tool in bounding the divergences and functional constants, for example Bakry-\'Emery curvature \parencite{bakry2006diffusions}, Ollivier-Ricci curvature \parencite{ollivier2009ricci}, and entropic-Ricci curvature \parencite{erbar2012ricci, erbar2018poincare}. Specifically in the problem of cutoff phenomenon, \parencite{salez2023cutoff} utilizes curvature to give a sufficient criterion for TV-cutoff.

\subsection{$\mathrm{TV}$-type $f$-divergences and R\'enyi divergence with $0<\alpha<1$}
\label{sec:TV-type divergences}
In this subsection, we will investigate a type of divergences, that we term as $\mathrm{TV}$-type $f$-divergences, in Definition \ref{def:TV-type divergence} below. These divergences will be shown to be equivalent to TV distance as well as R\'enyi divergence with $0<\alpha<1$ under cutoff phenomenon. We stress that results in this subsection have no requirement for reversibility.

\begin{definition}
    [TV-type $f$-divergences]
    \label{def:TV-type divergence}
     Given a convex function $f:[0,\infty)\rightarrow \mathbb R$ such that $f(1)=0$, the $f$-divergence $D_f\left(\cdot\|\cdot\right)$ is called a $\mathrm{TV}$-type $f$-divergence if there exists two continuous and strictly increasing functions $\psi_f, \Psi_f:[0,\infty)\rightarrow [0,\infty)$ with $\psi_f(0)=\Psi_f(0)=0$, such that for any two probability measures $\nu_1,\nu_2$ on any $\mathcal{\mathcal{X}}$ with $\nu_1\ll \nu_2$, 
     \begin{equation}\label{eq:definition of TV-type divergence}
         \psi_f\left(\mathrm{TV}(\nu_1\|\nu_2)\right)\leq D_f\left(\nu_1\|\nu_2\right)\leq \Psi_f\left(\mathrm{TV}(\nu_1\|\nu_2)\right).
     \end{equation}
\end{definition}

\begin{example}
    The upper bound of $\Psi_f$ can be readily identified using Proposition \ref{prop:properties of f-divergence} item \ref{item:convex conjugate of f}, and the lower bound $\psi_f$ can be found via comparison between $f(t)$ and $\frac{|t-1|}{2}$, the generator of total variation distance. Examples include:
    \begin{itemize}
        \item $\alpha$-divergence with $0<\alpha<1$: $f_\alpha (t)=\dfrac{t^\alpha-\alpha(t-1)-1}{\alpha-1}$, which satisfies
        \begin{equation*}
            f_\alpha(0)=1,\quad f_\alpha^*(0)=\lim_{u\rightarrow\infty}\frac{f_\alpha(u)}{u}=\frac{\alpha}{1-\alpha},
        \end{equation*}
        hence $f_\alpha(0)+f_\alpha^*(0)=\frac{1}{1-\alpha}$. Together with Pinsker's inequality in Proposition \ref{prop:properties of f-divergence} item \ref{item:Pinsker's inequality}, we have 
        \begin{align*}
            \psi_{f_\alpha}(s)&=\frac{1}{\alpha-1}\left(\exp\left(\frac{\alpha(\alpha-1)}{2}s^2\right)-1\right),\\
            \Psi_{f_\alpha}(s)&=\frac{s}{1-\alpha}.
        \end{align*}

        \item Squared Hellinger distance: $f(t)=\left(\sqrt{t}-1\right)^2$, by \parencite[Equation 8]{gibbs2002choosing} or \parencite{cam1972théorie}, we have 
        \begin{equation*}
            \psi_f(s)=s^2, \quad \Psi_f(s)=2s.
        \end{equation*} 

        \item Vincze-Le Cam distance: $f(t)=\dfrac{(t-1)^2}{t+1}$, which satisfies 
        \begin{equation*}
            f(0)=1, \quad f^*(0)=1,
        \end{equation*}
        hence $f(0)+f^*(0)=2$. Besides, by \eqref{eq:Le Cam distance and chi^2 divergence}, we have 
        \begin{align*}
            \mathrm{LC}(\nu_1,\nu_2)&=2\chi^2\left(\nu_1\bigg\|\frac{1}{2}\nu_1+\frac{1}{2}\nu_2\right)\geq 8\mathrm{TV}^2\left(\nu_1,\frac{1}{2}\nu_1+\frac{1}{2}\nu_2\right)\\
            &=2\mathrm{TV}^2(\nu_1,\nu_2),
        \end{align*}
        hence 
        \begin{equation*}
            \psi_f(s)=2s^2, \quad \Psi_f(s)=2s.
        \end{equation*}

        \item Jensen-Shannon divergence: $f(t)=t\ln t-(t+1)\ln \dfrac{t+1}{2}$, which satisfies 
        \begin{equation*}
            f(0)=\ln 2, \quad f^*(0)=\ln 2,
        \end{equation*}
        hence $f(0)+f^*(0)=2\ln 2$. Besides, by \eqref{eq:JS divergence and KL divergence} and Pinsker's inequality, we have
        \begin{align*}
            \mathrm{JS}(\nu_1\|\nu_2)&=\mathrm{KL}\left(\nu_1\bigg\|\frac{1}{2}\nu_1+\frac{1}{2}\nu_2\right)+\mathrm{KL}\left(\nu_2\bigg\|\frac{1}{2}\nu_1+\frac{1}{2}\nu_2\right)\\
            &\geq 2\mathrm{TV}^2\left(\nu_1,\frac{1}{2}\nu_1+\frac{1}{2}\nu_2\right)+2\mathrm{TV}^2\left(\nu_2,\frac{1}{2}\nu_1+\frac{1}{2}\nu_2\right)\\
            &=\mathrm{TV}^2(\nu_1,\nu_2),
        \end{align*}
        therefore 
        \begin{equation*}
            \psi_f(s)=s^2,\quad \Psi_f(s)=2s\ln 2.
        \end{equation*}
    \end{itemize}
\end{example}

The main result of this subsection demonstrates the equivalence between TV-type $f$-divergence cutoff and TV cutoff. 

\begin{theorem}
    [Equivalence between TV-type $f$-divergence cutoff and TV cutoff]
    \label{thm:equivalence between TV-type divergence cutoff and TV cutoff}
    Consider a sequence of Markov processes $\{X_t^{(n)},t\in T\}_{n=1}^\infty$ with state space $\mathcal{X}_n$, stationary distribution $\pi_n$ and semigroup $P_{t,n}$, where $P_{t,n}$ is \textbf{not necessarily reversible} on $L^2(\mathcal{X}_n,\pi_n)$ for each $n\geq 1$. Suppose $D_{f}\left(\cdot\|\cdot\right)$ is a $\mathrm{TV}$-type $f$-divergence, assume $\lim_{t\rightarrow\infty}\widetilde d_{\mathrm{TV},n}(t)=0$ and $\lim_{t\rightarrow\infty}\widetilde d_{f,n}(t)=0$, then the following statements hold.
    \begin{enumerate}[label=(\roman*)]
        \item\label{item:equivalence between TV-type and TV, 1}  If there exists $\{w_n\}_{n=1}^\infty$ such that for any $\varepsilon>0$, there is a $\left(\widetilde t_{\mathrm{TV},n}(\varepsilon), w_n\right)$ cutoff under $\widetilde d_{\mathrm{TV},n}(\cdot)$, then for any $\delta>0$, there is a $\left(\widetilde t_{f,n}(\delta), w_n\right)$ cutoff under $\widetilde d_{f,n}(\cdot)$.

        \item\label{item:equivalence between TV-type and TV, 2}  If there exists $\{w_n\}_{n=1}^\infty$ such that for any $\varepsilon>0$, there is a $\left(\widetilde t_{f,n}(\varepsilon), w_n\right)$ cutoff under $\widetilde d_{f,n}(\cdot)$, then for any $\delta>0$, there is a $\left(\widetilde t_{\mathrm{TV},n}(\delta), w_n\right)$ cutoff under $\widetilde d_{\mathrm{TV},n}(\cdot)$.
    \end{enumerate}
    Note that $\widetilde d_{\mathrm{TV},n}(t)$ and $\widetilde t_{\mathrm{TV},n}(\varepsilon)$ are the worst-case $\mathrm{TV}$ distance and mixing time respectively, while $\widetilde d_{f,n}(t)$ and $\widetilde t_{f,n}(\varepsilon)$ are defined in \eqref{eq:definition of worst-case f divergence} and \eqref{eq:definition of worst-case f divergence mixing time}.
\end{theorem}
\begin{proof}
    \ref{item:equivalence between TV-type and TV, 1}: According to \eqref{eq:definition of TV-type divergence}, for any $x\in \mathcal{X}_n$, 
    \begin{equation*}
        \psi_f\left(d_{\mathrm{TV}.n}(x,t)\right)\leq d_{f,n}(x,t)\leq \Psi_f\left(d_{\mathrm{TV}.n}(x,t)\right),
    \end{equation*}
    taking supremum over $x\in \mathcal{X}_n$ we have 
    \begin{equation*}
        \psi_f\left(\widetilde d_{\mathrm{TV}.n}(t)\right)\leq \widetilde d_{f,n}(t)\leq \Psi_f\left(\widetilde d_{\mathrm{TV}.n}(t)\right),
    \end{equation*}
    which implies for any given $\delta>0$,
    \begin{equation}\label{eq:equivalence between mixing time of TV-type divergence and TV distance, 1}
        \widetilde t_{\mathrm{TV},n}\left(\psi_f^{-1}(\delta)\right)\leq \widetilde t_{f,n}(\delta)\leq \widetilde t_{\mathrm{TV},n}\left(\Psi_f^{-1}(\delta)\right).
    \end{equation}
    Therefore, we have
    \begin{align*}
        \widetilde t_{f,n}(\delta)-\widetilde t_{f,n}(\eta)&\leq \widetilde t_{\mathrm{TV},n}\left(\Psi_f^{-1}(\delta)\right)-\widetilde t_{\mathrm{TV},n}\left(\psi_f^{-1}(\eta)\right), \quad \forall\hspace{0.1em} \eta>\delta,\\
        \widetilde t_{f,n}(\eta)-\widetilde t_{f,n}(\delta)&\leq \widetilde t_{\mathrm{TV},n}\left(\Psi_f^{-1}(\eta)\right)-\widetilde t_{\mathrm{TV},n}\left(\psi_f^{-1}(\delta)\right), \quad \forall\hspace{0.1em} \eta<\delta.
    \end{align*}
    Using \parencite[Proposition 2.3]{chen2008cutoff} and the condition of TV cutoff in \ref{item:equivalence between TV-type and TV, 1}, we have $w_n=o\left(\widetilde t_{\mathrm{TV},n}(\varepsilon)\right)$ for any $\varepsilon>0$, and
    recalling the big Oh notation $\mathcal{O}_\eta$ defined after \eqref{eq:O_eta first place}, we arrive at
    \begin{align*}
        \left|\widetilde t_{\mathrm{TV},n}\left(\Psi_f^{-1}(\delta)\right)-\widetilde t_{\mathrm{TV},n}\left(\psi_f^{-1}(\eta)\right)\right|&=\mathcal{O}_{\eta}(w_n),\quad \forall \hspace{0.1em}\eta>\delta,\\
        \left|\widetilde t_{\mathrm{TV},n}\left(\Psi_f^{-1}(\eta)\right)-\widetilde t_{\mathrm{TV},n}\left(\psi_f^{-1}(\delta)\right)\right|&=\mathcal{O}_{\eta}(w_n),\quad \forall\hspace{0.1em} \eta<\delta,
    \end{align*}
    and hence, for any $\eta>0$, 
    \begin{equation*}
        \left|\widetilde t_{f,n}(\delta)-\widetilde t_{f,n}(\eta)\right|=\mathcal{O}_{\eta}(w_n), 
    \end{equation*}
    which shows that for any $\delta>0$, there is a $\left(\widetilde t_{f,n}(\delta), w_n\right)$ cutoff under $\widetilde d_{f,n}(\cdot)$.

    \ref{item:equivalence between TV-type and TV, 2}: The proof is similar using an alternative of \eqref{eq:equivalence between mixing time of TV-type divergence and TV distance, 1}, i.e.
    \begin{equation*}
        \widetilde t_{f,n}\left(\Psi_f(\delta)\right)\leq \widetilde t_{\mathrm{TV},n}(\delta)\leq \widetilde t_{f,n}\left(\psi_f(\delta)\right),\quad \forall\hspace{0.1em} \delta>0.
    \end{equation*}
\end{proof}

\begin{corollary}
    [Extension to general divergences]
    \label{corollary:TV-type extension to Renyi divergence}
    If an information divergence $d\left(\cdot\|\cdot\right)$ satisfies
    \begin{itemize}
        \item for all $\nu_1,\nu_2$ on any $\mathcal{X}$, $d\left(\nu_1\|\nu_2\right)\geq 0$,

        \item for any $\nu_1,\nu_2$ on any $\mathcal{X}$, $d\left(\nu_1\|\nu_2\right)=0$ if and only if $\nu_1=\nu_2$, $\nu_2$-$a.e.$,

        \item there exists some $\mathrm{TV}$-type $f$-divergence $D_f\left(\cdot\|\cdot\right)$ and strictly increasing functions $\varphi_1,\varphi_2:[0,\infty]\rightarrow [0,\infty]$ with $\varphi_1(x),\varphi_2(x)<\infty$ if $x<\infty$, such that for all $\nu_1\ll\nu_2$ on any $\mathcal{X}$,
        \begin{equation*}
            \varphi_1\left(D_f\left(\nu_1\|\nu_2\right)\right)\leq d\left(\nu_1\|\nu_2\right)\leq \varphi_2\left(D_f\left(\nu_1\|\nu_2\right)\right),
        \end{equation*}
    \end{itemize}
    then $d\left(\cdot\|\cdot\right)$ and $\mathrm{TV}$ distance are equivalent under cutoff phenomenon, i.e. analogues of results in Theorem \ref{thm:equivalence between TV-type divergence cutoff and TV cutoff} also hold. Important examples include:
    \begin{itemize}
        \item R\'enyi divergence $R_\alpha$ with $0<\alpha<1$, which can be directly obtained via \eqref{eq:identity between alpha-divergence and Renyi divergence}.

        \item Bhattacharyya distance: $d_B(\nu_1,\nu_2):=-\ln \left(\int_\mathcal{X} \sqrt{d\nu_1d\nu_2}\right)$, and we have 
        \begin{equation*}
            \mathrm{Hel}^2(\nu_1,\nu_2)=2-2\exp\left(-d_B(\nu_1,\nu_2)\right).
        \end{equation*}
    \end{itemize}
    
\end{corollary}

\subsection{Separation-type divergences}
\label{sec:separation-type divergences}
In this subsection, in the setting of finite discrete-time Markov chains, we will consider the divergences which are equivalent to separation distance under cutoff phenomenon. Part of the motivation stems from \parencite{jonathan2016total} which states that separation cutoff and TV-cutoff are not equivalent for discrete-time lazy reversible Markov chains. As such we seek to find some divergences belonging to separation-type. For a Markov chain with transition matrix $P$ and positive stationary distribution $\pi$ on finite state space $\mathcal{X}$, the separation distance and separation mixing time are 
\begin{align*}
    d_{\mathrm{sep}}(x,t) &:=\max_{y\in\mathcal{X}}\left\{1-\frac{P^t(x,y)}{\pi(y)}\right\}, \quad t\in \mathbb N,\\
    \widetilde d_{\mathrm{sep}}(t) &:=\max_{x\in\mathcal{X}}d_{\mathrm{sep}}(x,t)=\max_{x,y\in\mathcal{X}}\left\{1-\frac{P^t(x,y)}{\pi(y)}\right\}, \quad t\in \mathbb N,\\
    \widetilde t_{\mathrm{sep}}(\varepsilon) &:=\inf\left\{t\in \mathbb N: \widetilde d_{\mathrm{sep}}(t)\leq \varepsilon\right\}, \quad t\in \mathbb N, \enspace \varepsilon>0.
\end{align*}

Using a similar argument as in Section \ref{sec:TV-type divergences}, the divergence $d(\cdot\|\cdot)$ is equivalent to separation distance under cutoff phenomenon if it satisfies 
\begin{itemize}
    \item For any discrete probability measures $\nu_1,\nu_2$ on $\mathcal{X}$, $d(\nu_1\|\nu_2)\geq 0$, and equality holds if and only if $\nu_1=\nu_2$.

    \item There exists two continuous and strictly increasing functions $\psi_f, \Psi_f:[0,\infty)\rightarrow [0,\infty)$ with $\psi_f(0)=\Psi_f(0)=0$, such that for any $t\in \mathbb N$, 
    \begin{equation}\label{eq:equivalence with separation distance}
        \psi_f\left(\widetilde d_{\mathrm{sep}}(t)\right)\leq \max_{x\in\mathcal{X}}d\left(P^t(x,\cdot)\|\pi\right)\leq \Psi_f\left(\widetilde d_{\mathrm{sep}}(t)\right).
    \end{equation}
\end{itemize}

\begin{example}
   An example of divergence that satisfies the above items is the reverse-$R_\infty$ divergence, and it suffices to check \eqref{eq:equivalence with separation distance}.
    \begin{itemize}
        \item Reverse-$R_\infty$: Recalling Proposition \ref{prop:properties of f-divergence} item \ref{item:monotonicity of Renyi divergence}, we define
        \begin{equation*}
            d_{R_\infty'}(x,t):=R_\infty\left(\pi\|P^t(x,\cdot)\right)=\max_{y\in\mathcal{X}} \hspace{0.1em}\ln \frac{\pi(y)}{P^t(x,y)},
        \end{equation*}
        and hence we have 
        \begin{equation*}
            d_{R_\infty'}(x,t)=\ln \frac{1}{1-d_{\mathrm{sep}}(x,t)},
        \end{equation*}
        where the right hand side is strictly increasing in $d_{\mathrm{sep}}(x,t)$.
    \end{itemize}
\end{example}


\subsection{Examples and counter examples}\label{subsec:examples}
In this subsection, we will verify the cutoff phenomenon in some classical reversible models under different types of divergences. More importantly, we use three counter examples, namely Aldous' example, Pak's example and product chains, to show that the classification of equivalence relationships in Table \ref{tab:merged} is natural and well-defined: the specific Markov processes exhibit cutoff in one type but not in another. Among these three examples, Aldous' example features classical construction and results. We also provide new results in Pak's example, while the model itself is classical. As for product chains, both the construction and results are new.  

We first clarify some notations used later. For two sequences $\{a_n\}$ and $\{b_n\}$, we recall that $a_n\sim b_n$ means $a_n/b_n\rightarrow 1$, and use $a_n=\Theta (b_n)$ to denote $a_n$ and $b_n$ are of same asymptotic order.

\begin{example}
    [Lazy random walk on hypercube]
    \label{eg:lazy random walk on hypercube}
    Suppose the $n^{th}$ process is the lazy random walk on hypercube $\{0,1\}^n$. At each step, we pick uniformly at random a coordinate from the $n$ coordinates and update it to $0$ and $1$ with probability $\frac{1}{2}$ respectively. According to \parencite[Example 12.16, Section 18.2]{levin2017markov}, these models exhibit $\mathrm{TV}$-cutoff, with spectral gap and mixing time satisfying
    \begin{equation*}
        \lambda_n=\frac{1}{n}, \quad \widetilde t_{\mathrm{TV},n}(\varepsilon)=\Theta (n\ln n).
    \end{equation*}
    Moreover, by \parencite[Example 3.7]{bobkov2006modified}, the modified log-Sobolev constant $\rho_n(1)$ is of order $\frac{1}{n}$, and recalling that
    \begin{equation*}
        \widetilde{t}_{2,n}(\varepsilon)\geq \widetilde{t}_{\mathrm{KL},n}(\varepsilon)\geq \widetilde{t}_{\mathrm{TV},n}\left(\sqrt{\frac{\varepsilon}{2}}\right),
    \end{equation*}
    then by Theorem \ref{thm:f divergence cutoff F_pq} to \ref{thm:KL divergence cutoff via LSI} and \ref{thm:equivalence between TV-type divergence cutoff and TV cutoff}, there is cutoff under $L^2$-type divergences, $\mathrm{TV}$-type divergences and $\mathrm{KL}$ divergence with cutoff window of the order $n$. 
\end{example}

\begin{example}
    [Aldous' example]\label{ex:Aldous}
    We consider the Aldous' example which is a reversible lazy random walk on finite state space, and readers can check \parencite[Example 7.1]{basu2017characterization}, \parencite[Section 6.1]{chen2008cutoff} or \parencite[Section 4.2]{chen2006cutoff} for more details. This model has no $\mathrm{TV}$-cutoff, and hence there is no cutoff for any $\mathrm{TV}$-type divergence by Theorem \ref{thm:equivalence between TV-type divergence cutoff and TV cutoff}. However, we still have 
    \begin{equation*}
        \liminf_{n\rightarrow\infty}\lambda_n>0, \quad \widetilde t_{\mathrm{TV},n}(\varepsilon)=\Theta (n),
    \end{equation*}
    which implies there is a cutoff for any $L^2$-type divergence by Theorem \ref{thm:f divergence cutoff F_pq} to \ref{thm:alpha-divergence cutoff with 1<alpha<2}. This also demonstrates that cutoff under $L^2$-type divergence is not equivalent to cutoff under $\mathrm{TV}$-type.
\end{example}

\begin{example}
    [Pak's example]\label{ex:Pak}
    Pak's example offers a regime of changing the pattern of mixing times and even destroying cutoff via perturbing the transition matrix. For its classical result introduced in \parencite[Example 18.7]{levin2017markov}, Pak's example has been used as a counter example for the sufficiency in Peres' conjecture under $\mathrm{TV}$-cutoff. In the following part, we will show that such argument can be extended to study cutoff under other divergences. Results include: (1) $\mathrm{TV}$-cutoff and $L^2$-cutoff are not equivalent (also mentioned in \parencite[Section 6.2]{chen2008cutoff}); (2) $\mathrm{TV}$-cutoff and $\mathrm{KL}$-cutoff are not equivalent; (3) Separation cutoff and $\mathrm{KL}/L^2$-cutoff are not equivalent.
    
    Suppose the $n^{th}$ Markov chain $\{X_t^{(n)}\}_{t\in \mathbb N}$ on finite state space $\mathcal{X}_n$ has reversible transition matrix $P_n$ and stationary distribution $\pi_n$. We stress the dependency on $P_n$ of various quantities of interests: we write $\widetilde d_{\mathrm{TV}}(P_n,t)$ to be the worst-case $\mathrm{TV}$ distance and $\widetilde d_{2}(P_n,t)$ as the worst-case $L^2$ distance at time $t$,  $\widetilde t_{\mathrm{TV}}(P_n,\varepsilon)$ and $\widetilde t_{2}(P_n,\varepsilon)$ as the $\mathrm{TV}$ and $L^2$ mixing time of $n^{th}$ chain respectively, the second largest singular value as $\kappa(P_n)$, and let $\lambda'(P_n):=-\ln \kappa(P_n)$. Assume $\{X_t^{(n)}\}_{t\in \mathbb N}$ has a $\mathrm{TV}$-cutoff, $\lambda'(P_n)\rightarrow 0$ and $\lambda'(P_n)\widetilde t_{\mathrm{TV}}(P_n, \varepsilon)\rightarrow\infty$, then $L^2$-cutoff also exists. Now, we consider another sequence of chain $\{Y_t^{(n)}\}_{t\in \mathbb N}$ on the same state space $\mathcal{X}_n$ with transition matrix 
    \begin{equation}\label{eq:definition of Q_n w.r.t. P_n}
        Q_n:=(1-c_n)P_n+c_n\Pi_n,\quad c_n \in (0,1),
    \end{equation}
    then $\pi_n$ is also the stationary distribution of $Q_n$, and we have 
    \begin{equation}\label{eq:expression of Q_t,n in Pak's example}
        Q_{t,n}=(1-c_n)^tP_{t,n}+\left(1-(1-c_n)^t\right)\Pi_n,
    \end{equation}
    which yields for $t\in \mathbb N$, 
    \begin{align}
        \widetilde d_{\mathrm{TV}}(Q_n,t)&=(1-c_n)^t\widetilde d_{\mathrm{TV}}(P_n,t),\label{eq:P_n and Q_n, TV}\\
        \widetilde d_{2}(Q_n,t)&=(1-c_n)^t\widetilde d_{2}(P_n,t),\label{eq:P_n and Q_n, L^2}\\
        \left\|Q_{t,n}-\Pi_n\right\|_{L^2\rightarrow L^2}&=(1-c_n)^t\left\|P_{t,n}-\Pi_n\right\|_{L^2\rightarrow L^2},\label{eq:P_n and Q_n, 2-norm}
    \end{align}
    and \eqref{eq:P_n and Q_n, 2-norm} together with Proposition \ref{prop: convergence rate of Markov semigroup} indicate that 
    \begin{equation}\label{eq:spectral gap of Q_n in Pak's example}
        \lambda'(Q_n)=\lambda'(P_n)-\ln (1-c_n).
    \end{equation}

    In the following part, we will explain why the new process $\{Y_t^{(n)}\}_{t\in \mathbb N}$ can serve as a counter-example. First of all, we assume $c_n\rightarrow 0$, and $\lambda'(P_n)c_n^{-1}\rightarrow \infty$, plugging into \eqref{eq:spectral gap of Q_n in Pak's example}, we have $\lambda'(Q_n)\sim \lambda'(P_n)$. If $c_n\widetilde t_{\mathrm{TV}}(P_n,\varepsilon)\rightarrow\infty$ for any $\varepsilon$, then by \eqref{eq:P_n and Q_n, TV} we have 
    \begin{equation}\label{eq:exact mixing time of Q_n in Pak's example}
        \widetilde t_{\mathrm{TV}}(Q_n,\varepsilon)\sim c_n^{-1}\ln \frac{1}{\varepsilon},
    \end{equation}
    and hence there is no cutoff for $Q_n$ under any $\mathrm{TV}$-type divergence. However, we have $\lambda'(Q_n)\widetilde t_{2}(Q_n,\varepsilon)\rightarrow\infty$ for any $\varepsilon>0$, hence cutoff still exists under $L^2$-type divergence for $Q_n$. 

    Here we give a rigorous proof for \eqref{eq:exact mixing time of Q_n in Pak's example}. Without loss of generality, for $\varepsilon<1$, when $n$ is very large, if there exists $M>1$ such that $\widetilde t_{\mathrm{TV}}(Q_n,\varepsilon)\geq M c_n^{-1}\ln \frac{1}{\varepsilon}=:t$, we have 
    \begin{align*}
        \varepsilon &\leq \widetilde d_{\mathrm{TV}}(Q_n,t)=\varepsilon^{-Mc_n^{-1}\ln (1-c_n)}\widetilde d_{\mathrm{TV}}(P_n,t)\\
        &\leq \varepsilon^{-Mc_n^{-1}\ln (1-c_n)}\rightarrow \varepsilon^M, \quad \text{as }n\rightarrow\infty,
    \end{align*}
    which is a contradiction. On the other hand, if there exists $0<m<1$ such that $\widetilde t_{\mathrm{TV}}(Q_n,\varepsilon)\leq m c_n^{-1}\ln \frac{1}{\varepsilon}=:t$, we have 
    \begin{align*}
        \varepsilon &\geq \widetilde d_{\mathrm{TV}}(Q_n,t)=\varepsilon^{-mc_n^{-1}\ln (1-c_n)}\widetilde d_{\mathrm{TV}}\left(P_n,m c_n^{-1}\ln \frac{1}{\varepsilon}\right)\\
        &\rightarrow \varepsilon^m, \quad \text{as }n\rightarrow\infty,
    \end{align*}
    where in the last step we have used $c_n^{-1}=o(\widetilde t_{\mathrm{TV}}(P_n,\eta))$ for any $\eta>0$, and that $P_n$ exhibits $\mathrm{TV}$-cutoff. This also forms a contradiction.

    Second, we further suppose the original chain $\{X_t^{(n)}\}_{t\in \mathbb N}$ is $\frac{2}{3}$-lazy (i.e. $P_n(x,x)\geq \frac{2}{3}$) for each $n$, and denote its spectral gap as $\lambda(P_n)$ and modified log-Sobolev constant as $\rho(P_n, 1)$. Since $c_n\rightarrow 0$, $\{Y_t^{(n)}\}_{t\in \mathbb N}$ is also lazy when $n$ is large. Recalling that for lazy reversible chain, the second largest singular value equals to the second largest eigenvalue, then by \eqref{eq:P_n and Q_n, 2-norm}, we have 
    \begin{equation*}
        1-\lambda(Q_n)=(1-c_n)(1-\lambda(P_n)),
    \end{equation*}
    and $c_n\rightarrow 0$ indicates $\lambda(Q_n)\sim \lambda(P_n)\sim \lambda'(P_n)$. We also assume that 
    \begin{equation}\label{eq:assumption of mLSI equivalent to PI in Pak's example}
        \rho(P_n,1)=\Theta(\lambda(P_n)),
    \end{equation}
    then plugging \eqref{eq:definition of Q_n w.r.t. P_n} into \eqref{eq:expression of Dirichelet f,g, reversible} and recalling $\rho(Q_n,1)=\mathcal{O}(\lambda(Q_n))$, we have $\rho(Q_n,1)=\Theta (\lambda(Q_n))$, hence $\rho(Q_n,1)\widetilde t_{\mathrm{KL}}(Q_n,\varepsilon)\rightarrow \infty$ by Pinsker's inequality.  Then by Theorem \ref{thm:KL divergence cutoff via LSI} there exists $\mathrm{KL}$-cutoff for $Q_n$, where we recall that $\widetilde t_{\mathrm{KL}}(Q_n,\varepsilon)$ is the worst-case $\mathrm{KL}$ mixing time. This implies that $\mathrm{TV}$-cutoff and $\mathrm{KL}$-cutoff are not equivalent. 
    
    The assumption \eqref{eq:assumption of mLSI equivalent to PI in Pak's example} can be easily verified. Recalling that the spectral gap and modified log-Sobolev constant remain the same order after lazifying the chain (i.e. $P_n\leftarrow \frac{1}{2}(P_n+I)$), we can list a few examples satisfying \eqref{eq:assumption of mLSI equivalent to PI in Pak's example}:
    \begin{itemize}
        \item Lazy random walk on hypercube: Example \ref{eg:lazy random walk on hypercube}.

        \item Lazified random transpositions: \parencite[Example 3.12]{bobkov2006modified}, \parencite{diaconis1981generating}, \parencite{diaconis1996cutoff}, \parencite[Corollary 3.1]{goel2004modified}, $\lambda(P_n), \rho(P_n,1)=\Theta\left(\frac{1}{n}\right)$, $\widetilde t_{\mathrm{TV}}(P_n,\varepsilon)=\Theta (n\ln n)$.

        \item Lazified high-temperature Curie-Weiss model with Glauber dynamics: \parencite[Theorem 12]{anari2021entropic}, \parencite[Theorem 1]{ding2009mixing}, for the fixed inverse temperature $\beta<1$, $\lambda(P_n), \rho(P_n,1)=\Theta\left(\frac{1}{n}\right)$, $\widetilde t_{\mathrm{TV}}(P_n,\varepsilon)=\Theta (n\ln n)$.
    \end{itemize}

    Third, under all the assumptions before, we further assume $P_n$ exhibits separation cutoff, and that for any $\varepsilon>0$,
    \begin{equation}\label{eq:separation mixing time equivalent to TV mixing time}
        \widetilde t_{\mathrm{sep}}(P_n,\varepsilon)=\Theta(\widetilde t_{\mathrm{TV}}(P_n,\varepsilon)),
    \end{equation} 
    hence $c_n\widetilde t_{\mathrm{sep}}(P_n,\varepsilon)\rightarrow\infty$. Recalling the separation distance defined in Section \ref{sec:separation-type divergences}, according to \eqref{eq:expression of Q_t,n in Pak's example}, we can write the separation distance of $Q_{t,n}$ to $\pi_n$ as 
    \begin{equation*}
        \widetilde d_{\mathrm{sep}}(Q_n,t)=(1-c_n)^t \widetilde d_{\mathrm{sep}}(P_n,t),
    \end{equation*}
    and since $\widetilde d_{\mathrm{sep}}(P_n,t)\leq 1$, similar to \eqref{eq:exact mixing time of Q_n in Pak's example}, we can still obtain
    \begin{equation*}
        \widetilde t_{\mathrm{sep}}(Q_n,\varepsilon)\sim c_n^{-1}\ln \frac{1}{\varepsilon},
    \end{equation*}
    hence there is no separation cutoff for $Q_n$. However, there are $\mathrm{KL}$-cutoff and $L^2$-cutoff according to the above discussions, and this implies that separation cutoff is not equivalent to $\mathrm{KL}$-cutoff or $L^2$-cutoff. Here assumption \eqref{eq:separation mixing time equivalent to TV mixing time} can be readily verified for example for the lazy random walk on hypercube \parencite[Theorem 18.8]{levin2017markov}.
    
\end{example}

\begin{example}
    [Product chains]
    \label{eg:product chains}
    Inspired by \parencite[Section 5.4]{su1995methods}, we provide a new approach of modifying the mixing time via extending the state space in the form of product chains, while maintaining reversibility. This possesses a similar form with lifted Markov chains \parencite{diaconis2000analysis, chen1999lifting}, although the latter one is non-reversible, and serves for different purposes like speed-up effects. For simplicity, in this article, we only consider the product chain consisting of a uniform random walk with an extra coordinate added. Under specific choices of parameters, we use it to show that $\mathrm{KL}$-cutoff and $L^2$-cutoff are not equivalent. For basic properties of product chains, readers can check \parencite{chen2018cutoffs}.
    
    For the $n^{th}$ process, we consider the continuous-time random walk on finite state space $\mathcal{X}_n=\mathcal{S}\times \mathcal{G}_n$, where $\mathcal{S}=\{0,1\}$, $0\in \mathcal{G}_n$ and $|\mathcal{G}_n|=:g_n<\infty$ as $n\rightarrow\infty$. At each step, according to a Poisson process with rate $1$, we pick the first coordinate $\mathcal{S}$ with probability $p_n<\frac{1}{2}$ and the second coordinate $\mathcal{G}_n$ with probability $1-p_n$, then take a random walk uniformly in that coordinate. The random walk on $\mathcal{S}$ has transition matrix $S(x,y)=\frac{1}{2}, \enspace \forall x,y\in S$, and the transition matrix on $\mathcal{G}_n$ is $G_n(x,y)=\frac{1}{g_n}, \enspace \forall x,y\in \mathcal{G}_n$. In view of these choices, the transition matrix for the $n^{th}$ process $\{X_t^{(n)}\}_{t\geq 0}$ is reversible, and can be written as
    \begin{align}
        P_{t,n}&=\exp\left(t\left(p_n S\otimes I+(1-p_n)I\otimes G_n-I\otimes I\right)\right)\label{eq:product chain example semigroup}\\
        &=e^{p_nt(S-I)}\otimes e^{(1-p_n)t(G_n-I)}, \quad t\in [0,\infty),\label{eq:product chain example decomposition}
    \end{align}
    and the stationary distribution is $\pi_n=\mu\otimes \nu_n$, where $\mu(x)=\frac{1}{2}$ and $\nu_n(x)=\frac{1}{g_n}$. Since both $S$ and $G_n$ have eigenvalues $0$ and $1$, by \eqref{eq:product chain example semigroup} and \parencite[Corollary 12.13]{levin2017markov}, the spectral gap of $\{X_t^{(n)}\}_{t\geq 0}$ is $\lambda_n=p_n$. 
    
    By symmetry, the distribution of $X_t^{(n)}$ at time $t$ is independent of the initial state, and without loss of generality we assume the initial state to be $X_0^{(n)}=(0,0)\in \mathcal{S}\otimes\mathcal{G}_n$. Denote $X_t^{(n)}=:\left(U_t^{(n)},V_t^{(n)}\right)$, then by \eqref{eq:product chain example decomposition}, $\{U_t^{(n)}\}_{t\geq 0}$ has transition matrix $e^{p_nt(S-I)}$ and $\{V_t^{(n)}\}_{t\geq 0}$ has transition matrix $e^{(1-p_n)t(G_n-I)}$, hence we have 
    \begin{equation*}
        \mathbb P\left(U_t^{(n)}=1\right)=\frac{1}{2}\left(1-e^{-p_nt}\right), \quad \mathbb P\left(U_t^{(n)}=0\right)=\frac{1}{2}\left(1+e^{-p_nt}\right),
    \end{equation*}
    and 
    \begin{align*}
        \mathbb P\left(V_t^{(n)}=x\right)&=\frac{1}{g_n}\left(1-e^{-(1-p_n)t}\right), \quad x\neq 0,\enspace x\in \mathcal{G}_n,\\
        \mathbb P\left(V_t^{(n)}=0\right)&=\frac{1}{g_n}\left(1+(g_n-1)e^{-(1-p_n)t}\right).
    \end{align*}
    
    Let $d_{\mathrm{KL},n}(X,t)$, $d_{\mathrm{KL},n}(U,t)$ and $d_{\mathrm{KL},n}(V,t)$ be the $\mathrm{KL}$ divergences of $X_t^{(n)}$, $U_t^{(n)}$ and $V_t^{(n)}$ to their stationary distributions $\pi_n$, $\mu$, $\nu_n$ respectively, then we have 
    \begin{align*}
        d_{\mathrm{KL},n}(U,t)&=\frac{1}{2}\left(1-e^{-p_nt}\right)\ln \left(1-e^{-p_nt}\right)+\frac{1}{2}\left(1+e^{-p_nt}\right)\ln \left(1+e^{-p_nt}\right),\\
        d_{\mathrm{KL},n}(V,t)&=\frac{g_n-1}{g_n}\left(1-e^{-(1-p_n)t}\right)\ln \left(1-e^{-(1-p_n)t}\right)\\
        &\enspace +\frac{1+(g_n-1)e^{-(1-p_n)t}}{g_n}\ln \left(1+(g_n-1)e^{-(1-p_n)t}\right),
    \end{align*}
    then the mixing time for $U_t^{(n)}$ can be written as 
    \begin{equation}\label{eq:KL mixing time of U_t}
        t_{\mathrm{KL},n}(U,\varepsilon)=\frac{1}{p_n}\phi(\varepsilon),
    \end{equation}
    where $\phi:(0,\infty)\rightarrow (0,\infty)$ is strictly decreasing and satisfies $\lim_{\varepsilon\rightarrow 0}\phi(\varepsilon)=\infty$ and $\lim_{\varepsilon\rightarrow\infty}\phi(\varepsilon)=0$. Moreover, using a similar argument in \parencite[Theorem 5.12]{su1995methods}, it is easy to check that when $g_n\rightarrow\infty$ is very large (e.g. $g_n\sim\exp(n^2)$), $d_{\mathrm{KL},n}(V,t)$ exhibits $\mathrm{KL}$-cutoff at cutoff time
    \begin{equation}\label{eq:KL mixing time of V_t}
        t_{\mathrm{KL},n}(V,\varepsilon)\sim \frac{1}{1-p_n}\ln \ln g_n.
    \end{equation}
    Similarly, for the $L^2$ distances, we have 
    \begin{align*}
        d_{2,n}(U,t)&=e^{-p_nt},\\
        d_{2,n}(V,t)&=\sqrt{g_n-1}\cdot e^{-(1-p_n)t},
    \end{align*}
    and the mixing times are 
    \begin{align}
        t_{2,n}(U,\varepsilon)&=\frac{1}{p_n}\ln \frac{1}{\varepsilon},\label{eq:L^2 mixing time of U_t}\\
        t_{2,n}(V,\varepsilon)&=\frac{1}{1-p_n}\left(\ln \frac{1}{\varepsilon}+\frac{1}{2}\ln (g_n-1)\right)\sim \frac{1}{2(1-p_n)}\ln g_n,\label{eq:L^2 mixing time of V_t}
    \end{align}
    then $V_t^{(n)}$ exhibits $L^2$-cutoff by \parencite[Proposition 2.3]{chen2008cutoff}. 

    Now, we take $g_n\rightarrow\infty$ sufficiently large (e.g. $g_n\sim \exp(n^2)$), and take $p_n=\left(\ln\ln g_n\right)^{-1}\rightarrow 0$, then we show that in this situation, for $\{X_t^{(n)}\}_{t\geq 0}$, there is $L^2$-cutoff but not $\mathrm{KL}$-cutoff. We first prove that for $\varepsilon<\phi^{-1}(1)$, the $\mathrm{KL}$ mixing time of $X_t^{(n)}$ is 
    \begin{equation}\label{eq:KL mixing time of X_t in product chain example}
        t_{\mathrm{KL},n}(X,\varepsilon)\sim \frac{1}{p_n}\phi(\varepsilon)=\phi(\varepsilon)\ln\ln g_n.
    \end{equation}
    Without loss of generality, if there exists $M>1$ such that $t_{\mathrm{KL},n}(X,\varepsilon)>\frac{M}{p_n}\phi(\varepsilon)$, then by the tensorization rule of $\mathrm{KL}$ divergence for product chains as stated in \parencite[Proposition 6]{barrera2006cut} or \parencite[Lemma A.4]{boursier2023universal}, we have 
    \begin{align*}
        \varepsilon &\leq d_{\mathrm{KL},n}\left(X,\frac{M}{p_n}\phi(\varepsilon)\right)=d_{\mathrm{KL},n}\left(U,\frac{M}{p_n}\phi(\varepsilon)\right)+d_{\mathrm{KL},n}\left(V,\frac{M}{p_n}\phi(\varepsilon)\right)\\
        &\leq \phi^{-1}\left(M\phi(\varepsilon)\right)+d_{\mathrm{KL},n}\left(V,\frac{M}{p_n}\right)\\
        &\rightarrow \phi^{-1}\left(M\phi(\varepsilon)\right)<\varepsilon,
    \end{align*}
    where the second inequality utilizes $\varepsilon<\phi^{-1}(1)$, and the third line follows from the cutoff time of $V_t^{(n)}$ in \eqref{eq:KL mixing time of V_t}. On the other hand, if there exists $0<m<1$ such that $t_{\mathrm{KL},n}(X,\varepsilon)<\frac{m}{p_n}\phi(\varepsilon)$, we have 
    \begin{align*}
        \varepsilon &\geq d_{\mathrm{KL},n}\left(X,\frac{m}{p_n}\phi(\varepsilon)\right)=d_{\mathrm{KL},n}\left(U,\frac{m}{p_n}\phi(\varepsilon)\right)+d_{\mathrm{KL},n}\left(V,\frac{m}{p_n}\phi(\varepsilon)\right)\\
        &\geq d_{\mathrm{KL},n}\left(U,\frac{m}{p_n}\phi(\varepsilon)\right)\\
        &=\phi^{-1}\left(m\phi(\varepsilon)\right)>\varepsilon,
    \end{align*}
    therefore \eqref{eq:KL mixing time of X_t in product chain example} is verified, and by \parencite[Proposition 2.3]{chen2008cutoff}, there is no $\mathrm{KL}$-cutoff for $\{X_t^{(n)}\}_{t\geq 0}$. Next, according to the tensorization rule of $L^2$ distance \parencite[Lemma A.4]{boursier2023universal}, we have 
    \begin{equation*}
        d_{2,n}^2(X,t)=\left(1+d_{2,n}^2(U,t)\right)\left(1+d_{2,n}^2(V,t)\right)-1\geq d_{2,n}^2(V,t),
    \end{equation*}
    then by \eqref{eq:L^2 mixing time of V_t}, we have 
    \begin{equation*}
        t_{2,n}(X,\varepsilon)\geq t_{2,n}(V,\varepsilon)\sim \ln g_n,
    \end{equation*}
    hence
    \begin{equation*}
        \lambda_n t_{2,n}(X,\varepsilon)\geq p_n t_{2,n}(V,\varepsilon)\sim \frac{\ln g_n}{\ln\ln g_n}\rightarrow\infty,
    \end{equation*}
    which implies there is $L^2$-cutoff for $\{X_t^{(n)}\}_{t\geq 0}$ by Proposition \ref{prop:characterization of L^p cutoff}.
\end{example}

\section{Non-reversible cases}

In this section, we will investigate the cutoff phenomenon of processes with non-reversible, or more generally, non-normal Markov generators. Non-reversibility arises naturally in many models as well as real-world applications. For example, non-reversible discrete-time algorithms appear as discretization of continuous-time reversible stochastic differential equations for sampling and optimization tasks, like \parencite{zhang2017hitting, vempala2019rapid, roberts1996exponential}. Moreover, for finite state space Markov chains, breaking reversibility can sometimes serve as an acceleration technique, see for instance \parencite{chen2013accelerating, chatterjee2021correction}. However, current results about cutoff phenomenon under non-reversible setting are still at primary stage, most of which only deal with some specific models, and a universal criterion is still quite open. For references of these non-reversible models, readers can check \parencite{bordenave2019cutoff, lancia2012entropy}.

Our aim is to give a common criterion to characterize non-reversible cutoff phenomenon. In Theorem \ref{thm:comparison between mixing times for L^p and R_alpha, non-reversible}, without assumption of reversibility or normality, we give new results of comparison between $L^p$-mixing times and $R_\alpha$-mixing times respectively, which serve as a complement to the classical results in \parencite[Proposition 5.1]{chen2008cutoff}. As its application, in Theorem \ref{thm:characterization of worst-case L^p cutoff with 1<p<infty, normal}, \ref{thm:characterization of worst-case alpha/Renyi-divergence cutoff, non-reversible} and \ref{thm:L^p cutoff, non-normal}, we extend \parencite[Theorem 5.3, 5.4]{chen2008cutoff} to the normal setting and the non-reversible setting generated by slight perturbations. 

\begin{theorem}
    [Comparison between $L^p/R_\alpha$-mixing times for $p\in(1,\infty)$, $\alpha\in (1,\infty)$, \textbf{non-reversible}]
    \label{thm:comparison between mixing times for L^p and R_alpha, non-reversible}
    Consider a sequence of Markov processes $\{X_t^{(n)},t\in T\}_{n=1}^\infty$ on state space $\mathcal{X}_n$ with stationary distribution $\pi_n$, and semigroup $P_{t,n}$, then the following statements hold.
    \begin{enumerate}[label=(\roman*)]
        \item\label{item:comparison between L^p mixing times, non-reversible} For any $p,q\in (1,\infty)$, there exists two positive strictly increasing mappings $\varepsilon\mapsto \psi_{p,q}(\varepsilon)$, $\varepsilon\mapsto \Psi_{p,q}(\varepsilon)$ and two constants $m_{p,q},M_{p,q}>0$, such that for any $\varepsilon>0$,
        \begin{equation}\label{eq:equivalence between t_p and t_q, normal}
            m_{p,q}\cdot\widetilde t_{p,n}\left(\psi_{p,q}(\varepsilon)\right)\leq \widetilde t_{q,n}(\varepsilon)\leq M_{p,q}\cdot\widetilde t_{p,n}\left(\Psi_{p,q}(\varepsilon)\right).
        \end{equation}

        \item\label{item:comparison between R_alpha mixing times, non-reversible} For any $\alpha,\beta\in (1,\infty)$, there exists two constants $c_{\alpha,\beta}$, $C_{\alpha,\beta}>0$, such that for any $\varepsilon>0$,
        \begin{equation}
            c_{\alpha,\beta}\cdot \widetilde t_{R_\alpha,n}(\varepsilon)\leq \widetilde t_{R_\beta,n}(\varepsilon)\leq C_{\alpha,\beta}\cdot \widetilde t_{R_\alpha,n}(\varepsilon).
        \end{equation}
    \end{enumerate}
\end{theorem}

\begin{proof}
    \ref{item:comparison between L^p mixing times, non-reversible}: For any $1<s<\infty$ and $r=\sqrt{s}$, for $x\in \mathcal{X}_n$, $g\in L^{s'}(\mathcal{X}_n,\pi_n)$ and $u,v>0$, let $\mu_{t,n}^x=\delta_x P_{t,n}$, by \eqref{eq:Kolmogorov Chapman}, we have 
    \begin{align*}
         \left|\left(\mu_{t,n}^x-\pi_n\right)(g)\right|&\leq d_{r,n}(x,u)\left\|(P_{v,n}-\Pi_n)(g)\right\|_{r'}\\
         &\leq d_{r,n}(x,u)\left\|P_{v,n}-\Pi_n\right\|_{L^{s'}\rightarrow L^{r'}},
    \end{align*}
    taking supremum over $x\in \mathcal{X}_n$, we have 
    \begin{equation}\label{eq:comparison between s and r=sqrt(s)}
        \widetilde d_{s,n}(u+v)\leq \widetilde d_{r,n}(u)\left\|P_{v,n}-\Pi_n\right\|_{L^{s'}\rightarrow L^{r'}}.
    \end{equation}
    By Riesz-Thorin Interpolation Theorem in Proposition \ref{prop:Riesz-Thorin theorem}, we have 
    \begin{align}
        \left\|P_{v,n}-\Pi_n\right\|_{L^{s'}\rightarrow L^{r'}}&\leq \left\|P_{v,n}-\Pi_n\right\|_{L^1\rightarrow L^1}^{1-\frac{1}{r}}\cdot \left\|P_{v,n}-\Pi_n\right\|_{L^{r'}\rightarrow L^\infty}^{\frac{1}{r}}\nonumber\\
        &\leq 2\cdot \widetilde d_{r,n}^{\frac{1}{r}}(v),\label{eq:another way to bound norm of s' to r'}
    \end{align}
    where the second inequality comes from $\left\|P_{v,n}-\Pi_n\right\|_{L^1\rightarrow L^1}\leq 2$. Plugging into \eqref{eq:comparison between s and r=sqrt(s)}, we have 
    \begin{align*}
        \widetilde d_{s,n}(u+v)&\leq \widetilde d_{r,n}(u)\left\|P_{v,n}-\Pi_n\right\|_{L^{s'}\rightarrow L^{r'}}\\
        &\leq 2\cdot \widetilde d_{r,n}(u)\cdot \widetilde d_{r,n}^{\frac{1}{r}}(v),
    \end{align*}
    taking $v=u$ yields 
    \begin{equation*}
        \widetilde d_{s,n}(2u)\leq 2\cdot \widetilde d_{\sqrt{s},n}^{1+\frac{1}{\sqrt{s}}}(u),
    \end{equation*}
    which implies for any $\delta>0$,
    \begin{equation*}
        \widetilde t_{s,n}\left(2\delta^{1+\frac{1}{\sqrt{s}}}\right)\leq 2\cdot\widetilde t_{\sqrt{s},n}(\delta).
    \end{equation*}
    Since the dynamics $x_{k+1}=\varphi (x_k)=\sqrt{x_k}$ converge to $1$, we get the result.

    \ref{item:comparison between R_alpha mixing times, non-reversible}: Let $h_{t,n}^x=\frac{d\delta_xP_{t,n}}{d\pi_n}$, for any $1<\alpha<\infty$, $\gamma=\sqrt{\alpha}$ and $u,v\geq 0$, similar to \eqref{eq:comparison between s and r=sqrt(s)}, we have 
    \begin{equation*}
        \left\|h_{u+v,n}^x\right\|_{\alpha}\leq \left\|h_{u,n}^x\right\|_{\gamma}\left\|P_{v,n}\right\|_{L^{\alpha'}\rightarrow L^{\gamma'}},
    \end{equation*}
    taking supremum over $x\in \mathcal{X}_n$, similar to \eqref{eq:another way to bound norm of s' to r'}, we have
    \begin{align*}
        \left\|P_{u+v,n}\right\|_{L^{\alpha'}\rightarrow \infty}&\leq  \left\|P_{u,n}\right\|_{L^{\gamma'}\rightarrow \infty}\left\|P_{v,n}\right\|_{L^{\alpha'}\rightarrow L^{\gamma'}}\\
        &\leq \left\|P_{u,n}\right\|_{L^{\gamma'}\rightarrow \infty}\left\|P_{v,n}\right\|_{L^{\gamma'}\rightarrow \infty}^{\frac{1}{\gamma}},
    \end{align*}
    and hence by taking $v=u$, 
    \begin{equation*}
        \left\|P_{2u,n}\right\|_{L^{\alpha'}\rightarrow \infty}\leq \left\|P_{u,n}\right\|_{L^{\gamma'}\rightarrow \infty}^{1+\frac{1}{\gamma}}.
    \end{equation*}
    Recalling \eqref{eq:expression of overline Renyi divergence}, we have 
    \begin{align*}
        \widetilde d_{R_\alpha,n}(2u)&\leq \frac{\alpha}{\alpha-1}\cdot \frac{\gamma+1}{\gamma}\cdot\frac{\gamma-1}{\gamma}\cdot\widetilde d_{R_\gamma,n}(u)\\
        &=\widetilde d_{R_\gamma,n}(u),
    \end{align*}
    which implies for any $\varepsilon>0$, 
    \begin{equation*}
        \widetilde t_{R_\alpha,n}(\varepsilon)\leq 2\cdot \widetilde t_{R_{\sqrt{\alpha}},n}(\varepsilon),
    \end{equation*}
    then similar to \ref{item:comparison between L^p mixing times, non-reversible}, we get the result.
\end{proof}

\begin{remark}
    Despite the requirement for reversibility, there is still a slight difference between Theorem \ref{thm:comparison between mixing times for L^p and R_alpha, non-reversible} and the results in Section \ref{sec:reversible} and \parencite[Proposition 5.1]{chen2008cutoff}: the case of $p=\infty$ and $\alpha=\infty$ can not be incorporated into the proof of the new results.
\end{remark}

\subsection{Normal cases}\label{subsec:normal}
In this subsection, we will investigate the case of normal processes, which refers to $P_tP_t^*=P_t^*P_t$ as discussed in Section \ref{sec:Markov process}. Although $P_t$ is not reversible, we still have the equalities in Proposition \ref{prop: convergence rate of Markov semigroup} which play a key role in quantifying lower bounds of mixing times.

Based on Theorem \ref{thm:comparison between mixing times for L^p and R_alpha, non-reversible}, in Theorem \ref{thm:characterization of worst-case L^p cutoff with 1<p<infty, normal}, we prove the equivalence of worst-case $L^p$-cutoff for $p\in (1,\infty)$ for normal processes, which is an extension of \parencite[Theorem 4.2, 4.3]{chen2008cutoff}. In Theorem \ref{thm:characterization of worst-case alpha/Renyi-divergence cutoff, non-reversible}, we proceed to present similar results for $\alpha$-divergence and R\'enyi divergence. 

\begin{theorem}
    [Characterization of worst-case $L^p$-cutoff with $1< p<\infty$, \textbf{normal}]
    \label{thm:characterization of worst-case L^p cutoff with 1<p<infty, normal}
    Consider a sequence of Markov processes $\{X_t^{(n)},t\in T\}_{n=1}^\infty$ on state space $\mathcal{X}_n$, stationary distribution $\pi_n$, spectral gap $\lambda_n$, and semigroup $P_{t,n}$, where $P_{t,n}$ is \textbf{normal} on $L^2(\mathcal{X}_n,\pi_n)$ for each $n\geq 1$. If $T=[0,\infty)$, let $\lim_{t\rightarrow\infty}\widetilde d_{p,n}(t)=0$ for each $1< p<\infty$, then the following statements hold:
    \begin{enumerate}[label=(E\arabic*)]
        \item\label{item:E1} There exists some $1<p<\infty$ and some $\varepsilon>0$ such that $\lambda_n\widetilde t_{p,n}(\varepsilon)\rightarrow\infty$.

        \item\label{item:E2} For any $1<p<\infty$ and any $\varepsilon>0$, $\lambda_n\widetilde t_{p,n}(\varepsilon)\rightarrow\infty$.

        \item\label{item:E3} For any $1<p<\infty$ and any $\varepsilon>0$, precutoff occurs.

        \item\label{item:E4} For any $1<p<\infty$ and any $\varepsilon>0$, cutoff occurs.

        \item\label{item:E5} For any $1<p<\infty$ and any $\varepsilon>0$, there is a $(\widetilde t_{p,n}(\varepsilon), \lambda_n^{-1})$ cutoff.
    \end{enumerate}
    
    If $T=\mathbb N$, assume $\lambda_n\rightarrow 0$, and that for some $1< p<\infty$ and $\varepsilon>0$, $\lim_{n\rightarrow \infty}\widetilde t_{p,n}(\varepsilon)=\infty$. If we substitute $\lambda_n'=\min \{1, \lambda_n\}$ into $\lambda_n$ in the items, then the statements above also hold. 
\end{theorem}
\begin{proof}
    We only consider $T=[0,\infty)$, and the case for $T=\mathbb N$ is similar. We only prove \ref{item:E1}$\Rightarrow$\ref{item:E2} and \ref{item:E3}$\Rightarrow$\ref{item:E2}, and the other proof  is same with Theorem \ref{thm:f divergence cutoff F_pq}.

    \ref{item:E3}$\Rightarrow$\ref{item:E2}: According to Proposition \ref{prop: convergence rate of Markov semigroup}, we have 
    \begin{equation*}
        \left\|P_{t,n}-\Pi_n\right\|_{L^2\rightarrow L^2}=e^{-\lambda_nt},
    \end{equation*}
    then use similar proof in \ref{item:B3}$\Rightarrow$\ref{item:B2}, we get the result.

    \ref{item:E1}$\Rightarrow$\ref{item:E2}: Suppose there exists some $1<p<\infty$ and $\varepsilon>0$ such that $\lambda_n\widetilde t_{p,n}(\varepsilon)\rightarrow\infty$. By \parencite[Corollary 2.5]{chen2008cutoff}, for any $\delta>0$, there is a $(\widetilde t_{p,n}(\delta), \lambda_n^{-1})$ cutoff, and hence $\lambda_n\widetilde t_{p,n}(\delta)\rightarrow\infty$. Then by Theorem \ref{thm:comparison between mixing times for L^p and R_alpha, non-reversible} item \ref{item:comparison between L^p mixing times, non-reversible}, for any $1<q<\infty$ and $\delta>0$, $\lambda_n\widetilde t_{q,n}(\delta)\rightarrow\infty$.
\end{proof}

Using similar argument and Theorem \ref{thm:comparison between mixing times for L^p and R_alpha, non-reversible} item \ref{item:comparison between R_alpha mixing times, non-reversible}, we directly get the following result.

\begin{theorem}
    [Characterization of worst-case $\alpha$-divergence and R\'enyi divergence cutoff with $1<\alpha<\infty$, \textbf{normal}]
    \label{thm:characterization of worst-case alpha/Renyi-divergence cutoff, non-reversible}
    Consider a sequence of Markov processes $\{X_t^{(n)},t\in T\}_{n=1}^\infty$ on finite state space $\mathcal{X}_n$, stationary distribution $\pi_n$, spectral gap $\lambda_n$ and semigroup $P_{t,n}$, where $P_{t,n}$ is \textbf{normal} on $L^2(\mathcal{X}_n,\pi_n)$ for each $n\geq 1$. 
    
    If $T=[0,\infty)$, for each $1< p<\infty$, let $\lim_{t\rightarrow\infty}\widetilde d_{f_\alpha,n}(t)=0$, then the following statement are equivalent:
    \begin{enumerate}[label=(F\arabic*)]
        \item\label{item:F1} There exists some $\alpha\in (1,\infty)$ and some $\varepsilon>0$ such that $\lambda_n\widetilde t_{f_\alpha,n}(\varepsilon)$ tends to infinity.

        \item\label{item:F2} For any $\alpha\in (1,\infty)$ and any $\varepsilon>0$, $\lambda_n\widetilde t_{f_\alpha,n}(\varepsilon)$ tends to infinity.

        \item\label{item:F3} For any $\alpha\in (1,\infty)$, precutoff occurs.

        \item\label{item:F4} For any $\alpha\in (1,\infty)$, cutoff occurs.

        \item\label{item:F5} For any $\alpha\in (1,\infty)$ and any $\varepsilon>0$, there is a $\left(\widetilde t_{f_\alpha,n}(\varepsilon),\lambda_n^{-1}\right)$ cutoff.
    \end{enumerate}
    
    If $T=\mathbb N$, assume $\lambda_n\rightarrow 0$, and that for some $1<\alpha<\infty$ and $\varepsilon>0$, $\lim_{n\rightarrow \infty}\widetilde t_{f_\alpha,n}(\varepsilon)=\infty$. If we substitute $\lambda_n'=\min \{1, \lambda_n\}$ into $\lambda_n$ in the items, then the statements above also hold. 

    If we replace $\widetilde t_{f_\alpha,n}(\cdot)$ with $\widetilde t_{R_\alpha,n}(\cdot)$ in the above statements, then the results still hold.
\end{theorem}

\subsection{Non-normal cases: A perturbation view}\label{subsec:nonnormal}
In this subsection, we shall consider non-normal Markov processes generated by suitably perturbing reversible processes. On a finite state space $\mathcal{X}$, we consider the Markov chain with transition matrix $P$ and its continuized chain with transition matrix $P_t=e^{t(P-I)}$. To obtain relatively tight bounds for mixing times under non-normal setting, we may need to use other quantities apart from the classical spectral gap $\lambda$ defined in Definition \ref{def:classical spectral gap}, which may meet trouble in obtaining the lower bound of mixing times without reversibility. An example can be found in \parencite[Page 106]{hermon2018technical}, where the worst-case TV mixing time is much smaller than the relaxation time $\lambda^{-1}$. There are already some techniques, like other way of defining the spectral gap of non-reversible processes, for example the multiplicative reversibilization in \parencite{fill1991eigenvalue}, pseudo-spectral gap in \parencite{paulin2015concentration} and Chatterjee's spectral gap in \parencite{chatterjee2025spectral}.

In the following part, we will use the eigenvalue of the second largest magnitude (resp.~real) part in discrete (resp.~continuous) time as the intermediate quantity in proving the equivalence of cutoff phenomenon between $L^p$ distances. 

\begin{proposition}
    [Lower bound in terms of eigenvalue for non-normal chains, \cite{montenegro2006mathematical} Theorem 4.9]
    \label{prop:lower bound in terms of eigenvalue, non-normal}
    For a finite state space Markov chain with transition matrix $P$, and its continuized chain $P_t=e^{t(P-I)}$, let $\beta_1$ be eigenvalue of $P$ with second largest magnitude, and $\gamma_1$ be the eigenvalue with second largest real part. Then, for the continuized chain,
    \begin{equation}\label{eq:lower bound in terms of eigenvalue, non-normal, continuous time}
        \widetilde d_1(t)\geq e^{-(1-\mathrm{Re}\hspace{0.1em}\gamma_1)t},\quad \widetilde t_1(\varepsilon)\geq \frac{1}{1-\mathrm{Re}\hspace{0.1em}\gamma_1}\ln \frac{1}{\varepsilon},\quad t\in [0,\infty),
    \end{equation}
    and for the discrete-time chain, 
    \begin{equation}\label{eq:lower bound in terms of eigenvalue, non-normal, discrete time}
        \widetilde d_1(k)\geq |\beta_1|^k,\quad \widetilde t_1(\varepsilon)\geq \frac{|\beta_1|}{1-|\beta_1|}\ln \frac{1}{\varepsilon}, \quad k\in \mathbb N.
    \end{equation}
\end{proposition}

Our motivation comes from a key observation. If a reversible transition matrix is slightly perturbed, then the coefficients of its characteristic polynomial have only minor changes, hence its spectrum should not change too much, and its spectral gap and $1-\mathrm{Re}\hspace{0.1em}\gamma_1$ or $1-|\beta_1|$ should be approximately the same. To rigorously give a perturbation bound, we will use the result from \parencite{cuenin2016non}.

\begin{definition}
    [$(Q,a,b)$-bounded perturbation]
    \label{def:(Q,a,b)-bounded perturbation}
    Let $Q$ be a self-adjoint linear operator on $L^2(\mathcal{X},\pi)$. We say the linear operator $A$ is $(Q,a,b)$-bounded if there exists $a,b\geq 0$ such that for any $f\in L^2(\mathcal{X},\pi)$, 
    \begin{equation}
        \left\|Af\right\|_2^2\leq a^2\left\|f\right\|_2^2+b^2\left\|Qf\right\|_2^2.
    \end{equation}
    
\end{definition}

\begin{proposition}
    [\cite{cuenin2016non}, Theorem 2.1, 2.12]
    \label{prop:perturbation results}
    Let $Q$ be a self-adjoint linear operator on $L^2(\mathcal{X},\pi)$, and $A$ is $(Q,a,b)$-bounded with $a\geq 0, 0\leq b<1$. Denote $\sigma(S)$ as the spectrum of any linear operator $S$, then the following statements hold.
    \begin{enumerate}[label=(\roman*)]
        \item\label{item:m multiplicity} Suppose $\lambda\in \sigma(Q)$ is an isolated eigenvalue with algebraic multiplicity $1\leq m<\infty$, set 
        \begin{equation*}
            \lambda_-:=\sup \{\nu\in \sigma(Q):\nu<\lambda\}, \quad \lambda_+:=\inf \{\nu\in \sigma(Q):\nu>\lambda\}.
        \end{equation*}
        If 
        \begin{equation}\label{eq:bounding condition for m multiplicity}
            \sqrt{a^2+b^2\lambda_-^2}+\sqrt{a^2+b^2\lambda^2}<\lambda-\lambda_-, \quad \sqrt{a^2+b^2\lambda^2}+\sqrt{a^2+b^2\lambda_+^2}<\lambda_+-\lambda,
        \end{equation}
        then the strip $\left\{z\in \mathbb C: \lambda-\sqrt{a^2+b^2\lambda^2}< \mathrm{Re}\hspace{0.1em} z< \lambda+\sqrt{a^2+b^2\lambda^2}\right\}$ contains exactly $m$ isolated eigenvalues of $Q+A$ (counted with algebraic multiplicity).

        \item\label{item:spectrum between hyperbolas} The spectrum $\sigma(Q+A)$ of $Q+A$ lies between hyperbolas:
        \begin{equation*}
            \sigma(Q+A)\subset \left\{z\in \mathbb C: |\mathrm{Im}\hspace{0.1em} z|^2\leq \frac{a^2+b^2|\mathrm{Re}\hspace{0.1em}z|^2}{1-b^2}\right\}.
        \end{equation*}
    \end{enumerate}
\end{proposition}

\begin{theorem}
    [Characterization of worst-case $L^p$-cutoff with $1<p<\infty$, \textbf{non-normal}]
    \label{thm:L^p cutoff, non-normal}
    Consider a sequence of Markov processes $\{X_t^{(n)},t\in T\}_{n=1}^\infty$ on finite state space $\mathcal{X}_n$, stationary distribution $\pi_n$, spectral gap $\lambda_n$, and semigroup $P_{t,n}$, where $P_{t,n}$ is \textbf{non-normal} on $L^2(\mathcal{X}_n,\pi_n)$ for each $n\geq 1$. 
    
    If $T=[0,\infty)$, suppose $P_{t,n}=e^{t(W_n-I)}$, where $W_n$ is a transition matrix, and $1-\lambda_n$ is an isolated eigenvalue of $\frac{W_n+W_n^*}{2}$. Let 
    \begin{equation}\label{eq:eta_n third largest eigenvalue}
    \eta_n:=\sup \left\{\nu\in \sigma\left(\frac{W_n+W_n^*}{2}\right): \nu<1-\lambda_n\right\},
    \end{equation}
    assume $\frac{W_n-W_n^*}{2}$ is $\left(\frac{W_n+W_n^*}{2},a_n,b_n\right)$-bounded with 
    \begin{equation}\label{eq:a_n,b_n bounded}
        \sqrt{a_n^2+b_n^2}<\frac{1}{2}\min \left\{1-\lambda_n-\eta_n, \lambda_n\right\}.
    \end{equation}
    Let $\lim_{t\rightarrow\infty}\widetilde d_{p,n}(t)=0$ for each $1< p<\infty$, then the following statements hold:
    \begin{enumerate}[label=(\roman*)]
        \item\label{item:sufficient condition for L^p, non-normal} For any $1<p<\infty$ and any $\varepsilon>0$, if $\lambda_n \widetilde t_{p,n}(\varepsilon)\rightarrow\infty$, then there is a $(\widetilde t_{p,n}(\varepsilon), \lambda_n^{-1})$ cutoff.

        \item\label{item:equivalence of L^p cutoff, non-normal} If for some $1<p<\infty$ and some $\varepsilon>0$, there is a $(\widetilde t_{p,n}(\varepsilon), \lambda_n^{-1})$ cutoff, then for any $1<q<\infty$ and any $\delta>0$, there is a $(\widetilde t_{q,n}(\delta), \lambda_n^{-1})$ cutoff. 
    \end{enumerate}

    If $T=\mathbb N$, suppose $P_{k,n}=W_n^k$, $1-\lambda_n$ is an isolated eigenvalue of $\frac{W_n+W_n^*}{2}$, and $\lambda_n\rightarrow 0$. We still denote $\eta_n$ as in \eqref{eq:eta_n third largest eigenvalue} and under the same assumption of \eqref{eq:a_n,b_n bounded}, we further assume that $b_n^2<\frac{3}{4}$ for any $n\geq 1$. Let $\lim_{t\rightarrow\infty}\widetilde d_{p,n}(t)=0$ for each $1< p<\infty$, and substitute $\lambda_n'=\min\{1,\lambda_n\}$ into $\lambda_n$ in the two items, then the statements above also hold. 
\end{theorem}

\begin{remark}
    The assumption of \eqref{eq:a_n,b_n bounded} can be readily verified in some instances. As a concrete example, suppose $U$ is a reversible transition matrix on $L^2(\mathcal{X},\pi)$, and transition matrix $V$ admits $\pi$ as its stationary distribution. Consider the following linear combination given by
    \begin{equation*}
        W:=(1-\varepsilon)U+\varepsilon V,\quad \varepsilon\in (0,1),
    \end{equation*}
    which can be non-normal. In order that $\frac{W-W^*}{2}$ is $\left(\frac{W+W^*}{2},a,b\right)$-bounded, we need to ensure
    \begin{equation}\label{eq:assumption in terms of U and V}
        \left\|\varepsilon \frac{V-V^*}{2}f\right\|_2^2\leq a^2\left\|f\right\|_2^2+b^2\left\|\left((1-\varepsilon)U+\varepsilon  \frac{V+V^*}{2}\right)f\right\|_2^2, \quad \forall f\in L^2(\mathcal{X},\pi).
    \end{equation}
    Since the left hand side above is smaller than $\varepsilon^2\left\|f\right\|_2^2$, after fixing $a$ and $b$ which satisfy \eqref{eq:a_n,b_n bounded}, we can take any $0<\varepsilon\leq a$ and \eqref{eq:assumption in terms of U and V} holds. 
\end{remark}

\begin{proof}
    It suffices to prove item \ref{item:equivalence of L^p cutoff, non-normal}. We first consider the continuous-time case. Denote $\gamma_{1,n}$ as the eigenvalue of $W_n$ with second largest real part. Since $W_n$ can be viewed as adding a perturbation to a reversible transition matrix: 
    \begin{equation*}
        W_n=\frac{W_n+W_n^*}{2}+\frac{W_n-W_n^*}{2},
    \end{equation*}
    then by condition \eqref{eq:a_n,b_n bounded}, we have 
    \begin{gather*}
        \sqrt{a_n^2+b_n^2\eta_n^2}+\sqrt{a_n^2+b_n^2(1-\lambda_n)^2}\leq 2\sqrt{a_n^2+b_n^2}<1-\lambda_n-\eta_n,\\
        \sqrt{a_n^2+b_n^2(1-\lambda_n)^2}+\sqrt{a_n^2+b_n^2}\leq 2\sqrt{a_n^2+b_n^2}<\lambda_n,
    \end{gather*}
    then by Proposition \ref{prop:perturbation results} item \ref{item:m multiplicity}, we have 
    \begin{align*}
        1-\mathrm{Re}\hspace{0.1em}\gamma_{1,n}&\leq 1-\left((1-\lambda_n)-\sqrt{a_n^2+b_n^2(1-\lambda_n)^2}\right)\\
        &\leq \lambda_n+\sqrt{a_n^2+b_n^2}\\
        &\leq \frac{3}{2}\cdot\lambda_n,
    \end{align*}
    together with \eqref{eq:lower bound in terms of eigenvalue, non-normal, continuous time}, for any $\varepsilon_1>0$,
    \begin{equation}\label{eq:lower bound of mixing time in terms of spectral gap, L^p, non-normal}
        \widetilde t_{p,n}(\varepsilon_1)\geq \widetilde t_{1,n}(\varepsilon_1)\geq \frac{2}{3\lambda_n}\ln \frac{1}{\varepsilon_1},
    \end{equation}
    hence
    \begin{equation}\label{eq:lower bound of ratio, non-normal, continuous time}
        \frac{\widetilde t_{p,n}(\varepsilon_1)}{\widetilde t_{p,n}(1/4)}\geq \frac{2}{3\lambda_n\widetilde t_{p,n}(1/4)}\ln \frac{1}{\varepsilon_1}.
    \end{equation}
    The condition in item \ref{item:equivalence of L^p cutoff, non-normal} implies precutoff occurs for some $1<p<\infty$, then if 
    \begin{equation}\label{eq:contradiction of not to infty}
        \liminf_{n\rightarrow\infty}\lambda_n\widetilde t_{p,n}(1/4)=c<\infty,
    \end{equation}
    plugging into \eqref{eq:lower bound of ratio, non-normal, continuous time} leads to
    \begin{equation*}
        \limsup_{n\rightarrow\infty}\frac{\widetilde t_{p,n}(\varepsilon_1)}{\widetilde t_{p,n}(1/4)}\geq \frac{2}{3c}\ln \frac{1}{\varepsilon_1},
    \end{equation*}
    take $\varepsilon_1\rightarrow 0$, according to \parencite[Proposition 2.3]{chen2008cutoff}, we get the contradiction of \eqref{eq:contradiction of not to infty}. Therefore, for any $\delta>0$, $\lambda_n\widetilde t_{p,n}(\delta)\rightarrow\infty$. Then according to Theorem \ref{thm:comparison between mixing times for L^p and R_alpha, non-reversible} item \ref{item:comparison between L^p mixing times, non-reversible}, for any $1<q<\infty$, $\lambda_n\widetilde t_{q,n}(\delta)\rightarrow\infty$, and hence there is a $(\widetilde t_{q,n}(\delta), \lambda_n^{-1})$ cutoff.

    Next, we consider the discrete-time case. Denote $\beta_{1,n}$ as the eigenvalue of $W_n$ with second largest magnitude. Similar to the argument above, and according to Proposition \ref{prop:perturbation results} item \ref{item:spectrum between hyperbolas}, we have 
    \begin{align*}
        1-|\beta_{1,n}|&\leq |1-\beta_{1,n}|\leq 1-\mathrm{Re}\hspace{0.1em}\beta_{1,n}+|\mathrm{Im}\hspace{0.1em}\beta_{1,n}|\\
        &\leq 1-\left((1-\lambda_n)-\sqrt{a_n^2+b_n^2(1-\lambda_n)^2}\right)+\sqrt{\frac{a_n^2+b_n^2}{1-b_n^2}}\\
        &\leq \lambda_n+3\sqrt{a_n^2+b_n^2}\\
        &\leq 4\lambda_n,
    \end{align*}
    where in the third inequality we have used $b_n^2<\frac{3}{4}$. Since $\lambda_n\rightarrow 0$, we have $|\beta_{1,n}|\rightarrow 1$. Then by \eqref{eq:lower bound in terms of eigenvalue, non-normal, discrete time}, when $n$ is sufficiently large, for any $\varepsilon_1>0$,
    \begin{equation*}
        \widetilde t_{p,n}(\varepsilon_1)\geq \widetilde t_{1,n}(\varepsilon_1)\geq \frac{1}{8\lambda_n}\ln \frac{1}{\varepsilon_1},
    \end{equation*}
    which is similar to \eqref{eq:lower bound of mixing time in terms of spectral gap, L^p, non-normal}, then we get the result. 
\end{proof}

\begin{remark}
    Theorem \ref{thm:L^p cutoff, non-normal} indicates that under the assumption \eqref{eq:a_n,b_n bounded}, $L^p$-cutoff are equivalent to $L^2$-cutoff for $1<p<\infty$. Moreover, for $\alpha$-divergence and R\'enyi divergence, the proof and result are similar, and we can obtain that they are also $L^2$-type divergences under cutoff phenomenon.
\end{remark}

\section*{Acknowledgements}
Youjia Wang gratefully acknowledges the financial support from National University of Singapore via the Presidential Graduate Fellowship. Michael Choi acknowledges the financial support of the project “MAPLE: Mechanistic Accelerated Prediction of Protein Secondary Structure via LangEvin Monte Carlo” with grant number 22-5715-P0001 under the NUS Faculty of Science Ministry of Education Tier 1 grant Data for Science and Science for Data collaborative scheme, project NUSREC-HPC-00001 and NUSREC-CLD-00001 for NUS HPC-AI Priority Projects for Research Program, as well as the startup funding of the National University of Singapore with grant number A-0000178-01-00.

\printbibliography

\end{document}